\documentclass[10pt,leqno,twoside]{article}%
\usepackage{amssymb,bbm,amsmath,amsfonts,amsthm, amscd}
\usepackage{mathtools}
\usepackage{stmaryrd}

\usepackage[utf8]{inputenc}
\usepackage[T1]{fontenc}
\usepackage[french, english]{babel}

\usepackage[style=alphabetic, url=false, doi=false]{biblatex}
\usepackage[right=3cm, left=3cm]{geometry}

\usepackage{asymptote} 

\usepackage[dvipsnames]{xcolor}

\usepackage[all]{xy} 
 
\usepackage{tikz}
\usetikzlibrary{calc}
\usetikzlibrary{decorations.pathreplacing}
\usetikzlibrary{shapes.misc}
\usetikzlibrary{decorations.pathmorphing,patterns}
\usetikzlibrary{arrows}
\usetikzlibrary {arrows.meta}

\usepackage{pgfplots}
\pgfplotsset{compat=1.15}

\usepackage{mathrsfs} 
\usepackage{yfonts} 
\usepackage{dsfont} 

\usepackage[toc,page]{appendix}

\usepackage[lined,ruled]{algorithm2e}

\usepackage{ifthen}

\def\R {\mathbb{R}}
\def\C {\mathbb{C}}
\def\Q {\mathds{Q}}

\def\N {\mathds{N}}

\def\eps {\varepsilon}

\DeclareMathOperator{\Prob}{\mathds{P}}
\DeclareMathOperator{\Esp}{\mathbb{E}}

\DeclareMathOperator{\diag}{Diag}

\DeclareMathOperator{\vect}{Span}

\DeclareMathOperator{\ima}{Im}  

\DeclareMathOperator{\Id}{Id}   

\DeclareMathOperator*{\argmax}{arg\, \max}

\DeclareMathOperator{\sgn}{sgn}  

\DeclareMathOperator{\re}{\mathfrak{Re}}
\newcommand*{\transp}[1]{\ensuremath{\prescript{t}{}{#1}}}

\newcommand{\BEQ} {\begin{equation} }
  \newcommand{\EEQ} {\end{equation} }

\newcommand{\BTHM}{\begin{theorem}}
  \newcommand{\ETHM}{\end{theorem}}
  
  \newcommand{\BL}{\begin{lemma}}
  \newcommand{\EL}{\end{lemma}}

\newtheorem{theorem}{\textsc{Theorem}}[section]
\newtheorem{lemma}[theorem]{\textsc{Lemma}}

\newtheorem*{theorem*}{\textsc{Theorem}}
\newtheorem*{lemma*}{\textsc{Lemma}}
\newtheorem*{corollary*}{\textsc{Corollary}}
\newtheorem*{prop*}{\textsc{Proposition}}
\theoremstyle{definition}
\newtheorem{assumption}{\textsc{Assumption}}[section]
\newtheorem*{assumption*}{\textsc{Assumption}}

\newtheorem*{example*}{\textsc{Example}}

\newtheorem*{examples*}{\textsc{Examples}}

\newtheorem{remark}{\textsc{Remark}}[section]
\newtheorem*{remark*}{\textsc{Remark}}
\newtheorem*{notation}{\textsc{Notation}}
\newtheorem*{not&def}{\textsc{Notation and definitions}}

\newtheorem*{app*}{\textsc{Application}}
\newtheorem{proper}{\textsc{Property}}[section]
\newtheorem*{proper*}{\textsc{Property}}

\newtheorem*{propers*}{\textsc{Properties}}

\newcommand{\limn}[1]{\underset{n \to +\infty}{\lim}}
\newcommand{\limx}[1]{\underset{x \to +\infty}{\lim}}
\newcommand{\limt}[1]{\underset{t \to +\infty}{\lim}}
\newcommand{\tendto}[1]{\mathop{\longrightarrow}\limits_{#1 \to +\infty}}

\newcommand{\suite}[2][n]{( #2_{#1} )_{#1 \in \N} }
\newcommand{\suit}[2][n]{( #2 )_{#1 \in \N} }

\newcommand{\pat}[1]{\left( #1 \right)}
\newcommand{\bigpat}[1]{\bigl( #1 \bigr)}
\newcommand{\Bigpat}[1]{\Bigl( #1 \Bigr)}

\newcommand{\Biggpat}[1]{\Biggl( #1 \Biggr)}
\newcommand{\cro}[1]{\left[ #1 \right]}

\newcommand{\acc}[1]{\left\{ #1 \right\}}
\newcommand{\bigacc}[1]{\bigl\{ #1 \bigr\}}

\newcommand{\dua}[1]{\left< #1 \right>}

\newcommand{\abs}[1]{\left| #1 \right|}

\newcommand{\bigabs}[1]{\bigl| #1 \bigr|}
\newcommand{\nor}[2]{\left\| {#1} \right\|_{#2}}
\newcommand{\bignor}[2]{\bigl\| {#1} \bigr\|_{#2}}

\newcommand{\vertiii}[1]{{\left\vert\kern-0.25ex\left\vert\kern-0.25ex\left\vert #1 \right\vert\kern-0.25ex\right\vert\kern-0.25ex\right\vert}}

\DeclarePairedDelimiterX{\inner}[2]{\langle}{\rangle}{#1, #2}

\newcommand{\eint}[1]{\llbracket #1 \rrbracket}

\def\nbQ{{\mathchoice {\setbox0=\hbox{$\displaystyle\rm
Q$}\hbox{\raise
0.15\ht0\hbox to0pt{\kern0.4\wd0\vrule height0.8\ht0\hss}\box0}}
{\setbox0=\hbox{$\textstyle\rm Q$}\hbox{\raise
0.15\ht0\hbox to0pt{\kern0.4\wd0\vrule height0.8\ht0\hss}\box0}}
{\setbox0=\hbox{$\scriptstyle\rm Q$}\hbox{\raise
0.15\ht0\hbox to0pt{\kern0.4\wd0\vrule height0.7\ht0\hss}\box0}}
{\setbox0=\hbox{$\scriptscriptstyle\rm Q$}\hbox{\raise
0.15\ht0\hbox to0pt{\kern0.4\wd0\vrule height0.7\ht0\hss}\box0}}}}

\def\nbZ{{\mathchoice {\hbox{$\sf\textstyle Z\kern-0.4em Z$}}
{\hbox{$\sf\textstyle Z\kern-0.4em Z$}}
{\hbox{$\sf\scriptstyle Z\kern-0.3em Z$}}
{\hbox{$\sf\scriptscriptstyle Z\kern-0.2em Z$}}}}

\def\ocirc#1{\ifmmode\setbox0=\hbox{$#1$}\dimen0=\ht0
    \advance\dimen0 by1pt\rlap{\hbox to\wd0{\hss\raise\dimen0
    \hbox{\hskip.2em$\scriptscriptstyle\circ$}\hss}}#1\else
    {\accent"17 #1}\fi}

\def\nbZ{{\mathchoice {\hbox{$\sf\textstyle Z\kern-0.4em Z$}}
{\hbox{$\sf\textstyle Z\kern-0.4em Z$}}
{\hbox{$\sf\scriptstyle Z\kern-0.3em Z$}}
{\hbox{$\sf\scriptscriptstyle Z\kern-0.2em Z$}}}}

\usepackage{enumerate}
\usepackage{enumitem}

\usepackage[linktocpage,colorlinks=true,raiselinks]{hyperref}
\hypersetup{urlcolor=blue, citecolor=red, linkcolor=blue}

\usepackage{authblk}

\makeatletter
\@addtoreset{equation}{section}

\makeatother

\DeclareMathAlphabet{\mathpzc}{OT1}{pzc}{m}{it}

\def\holder {\mathfrak{C}}  
\newcommand{\holdern}[1]{\nor{#1}{\beta, \infty}}  

\def\dhaus {{d_{\mathrm{Haus}}}}
\def\dirhaus {{d_{\overrightarrow{\mathrm{Haus}}}}}

\def\dint {{d^*_n}}  
\def\din  {{d^*}}
\def\bdint {{\mathbf{\dint}}} 
\def\dO {{\mathbf{\din}}}
\def\cint  {{C_{\din}}}  
\def\dintup {{\overline{\dint}}}  
\def\bdintup {{\mathbf{\dintup}}}
\def\fint {{g^*_n}}  
\def\fin  {{g^*}}

\def\sphere {{\mathbb{S}}}  
\def\optsub {{\mathcal{S}}}  

\def\balpha {{\boldsymbol \alpha}}

\def\v  {\bold v}

\def\b  {\bold b}

\def\epsu {{\underline{\eps}}}

\def\U  {\mathbb{U}}

\def\H  {\mathbb{H}}
\def\B  {\mathbb{B}}
\def\L  {\mathcal{L}}
\def\O  {\mathcal{O}}
\def\Q  {\mathcal{Q}}
\def\G  {\mathcal{G}}
\def\Wtilde {\mathrlap{\, \widetilde{\phantom{A}}}W}
\newcommand{\Bon}[2]{\mathcal{B}_{\mathrm{on}}^{#1}(#2)}

\newif\ifrem
\newif\ifsuppcalc

\usepackage{fancyhdr}
\usepackage{setspace}

\usepackage{changepage}

\usepackage{titlesec}
\makeatletter
\titleformat{\paragraph}[runin]{\normalsize\itshape\bfseries}{\theparagraph}{1em}{}[\@addpunct{.}]
\titleformat{\subparagraph}[runin]{\normalsize\itshape}{\theparagraph}{1em}{}[\@addpunct{.}]
\makeatother

\begin{document}


\title{\LARGE \textbf{\textsc{Contraction rates and projection subspace estimation with Gaussian process priors in high dimension}}}

\author[1]{\textbf{Elie \textsc{Odin}} \thanks{\texttt{elie.odin@math.univ-toulouse.fr}}}
\author[1, 2, 4]{\textbf{François \textsc{Bachoc}} \thanks{\texttt{francois.bachoc@math.univ-toulouse.fr}}}
\author[3, 4]{\textbf{Agnès \textsc{Lagnoux}} \thanks{\texttt{lagnoux@univ-tlse2.fr}}}

\affil[1]{Institut de Mathématiques de Toulouse; UMR5219. Université de Toulouse; CNRS. UT3, F-31062 Toulouse, France. \vspace{0.1cm}}
\affil[2]{Institut Universitaire de France (IUF), France.  \vspace{0.1cm}}
\affil[3]{Institut de Mathématiques de Toulouse; UMR5219. Université de Toulouse; CNRS. UT2J, F-31058 Toulouse, France.\vspace{0.1cm}}
\affil[4]{Artificial and Natural Intelligence Toulouse Institute (ANITI), Toulouse, France.}

\date{07 2025}

\maketitle

\bigskip
\smallskip

\begin{abstract}
This work explores the dimension reduction problem for Bayesian nonparametric regression and density estimation. 
More precisely, we are interested in estimating a functional parameter $f$ over the unit ball in $\mathbb{R}^d$, which depends only on a $\din$-dimensional subspace of $\mathbb{R}^d$, with $\din < d$.
It is well-known that rescaled Gaussian process priors over a given function space achieve smoothness adaptation and posterior contraction with near minimax-optimal rates. Furthermore, hierarchical extensions of this approach, equipped with subspace projection, can also adapt to the intrinsic dimension $\din$ (\cite{Tokdar2011DimensionAdapt}).
When the ambient dimension $d$ does not vary with $n$, the minimax rate remains of the order $n^{-\beta/(2\beta +\din)}$, where $\beta$ denotes the smoothness of $f$.
However, this is up to multiplicative constants that can become prohibitively large as $d$ increases. The dependence between the contraction rate and the ambient dimension has not been fully explored yet and this work provides a first insight: we let the dimensions $\din$ and $d$ grow with $n$, and by combining the arguments of \cite{Tokdar2011DimensionAdapt} and \cite{Castillo2024DHGP}, we derive growth rates for them that still lead to posterior consistency with minimax rate.
We also discuss the optimality of the growth rate for $d$.
Additionally, we provide a set of assumptions under which a consistent estimation of $f$ leads to correct estimation of the subspace projection, assuming that $\din$ is known.
\end{abstract}

\bigskip

\section{Introduction}


We are interested in the nonparametric problem of estimating an unknown 
function $f : \U_d \subset \R^d \to \R$ with a high-dimensional input space, where $\U_d$ is the unit ball. For both regression and density estimation, it is well known that the minimax rates without sparsity assumptions are of order $n^{-\beta/(2\beta + d)}$, where $\beta$ is the smoothness of $f$ and $n$ is the sample size (\cite{Birge1986DensityHelling}, \cite{Stone1982Rates}). 
This results in an extremely slow convergence when $d$ is large, a phenomenon known as the curse of dimensionality. 
To counteract this effect, it is usually assumed that $f$ depends only on a low dimensional structure in order to obtain a rate $n^{-\beta/(2\beta + \din)}$ that depends only on the \textit{intrinsic dimension} $\din$ of this structure (\cite{Lafferty2008Rodeo}, \cite{Yang2015Minimax}, \cite{Jiang2021VariableSelection}).
In this article, we 
focus on the (semi-)parametric case where $f$ depends only on a $\din$-dimensional linear subspace $\mathcal{S}_n \subset \R^d$, with $\din \ll d$ (\cite{Tokdar2010Bayesian}, \cite{Tokdar2011DimensionAdapt}).
As the above rates are given up to a multiplicative constant, which may itself depend on the \textit{ambient dimension} $d$, another problem arises: 
establishing if the number $n$ of available data is sufficient in regard to the problem's dimension. This is generally done by allowing the ambient dimension $d = d_n$ and the intrinsic dimension $\din = \dint$ to grow with $n$, and then by observing which growth rate still permits near minimax estimation at rate of order $n^{-\beta/(2\beta + \dint)}$.

We distinguish two cases, whether the subspace $\mathcal{S}_n$ is parallel to the axes or not.
In the first case (when $\mathcal{S}_n$ is parallel to the axes), the dimension reduction problem is referred to as \textit{variable selection}. In this context, it is known that for nonparametric regression, the sparsity pattern can be consistently recovered when $d_n$ grows exponentially with the sample size (\cite{Comminges2012tight}, \cite{Yang2015Minimax}). More precisely, \cite{Comminges2012tight} show that there exist positive constants $c_1, c_1'$, and $c_* < c^*$ such that
\begin{itemize}
\item if $e^{c_1 \dint} \log (d_n/\dint) \cdot n^{-1} < c_*$, there exists a consistent estimator of the sparsity pattern, 
\item if $e^{c_1' \dint} \log (d_n/\dint) \cdot n^{-1} > c^*$, no such estimator exists.
\end{itemize}
This phase transition phenomenon seems to be similar in the linear regression framework, replacing $e^\dint$ by $\dint$ (see \cite{Verzelen2012UHD} and \cite{Wainwright2009sharp}).
The minimax rate for the estimation of $f$ then takes the composite form $n^{-\beta/(2\beta + \dint)} + \sqrt{\dint \log(d_n/\dint)/n}$, as shown in \cite{Yang2015Minimax}, where the first term is the minimax rate for the intrinsic part of $f$ and the second term is the variable selection uncertainty.

In the second case (when nothing is assumed on $\mathcal{S}_n$), the estimation of a minimal subspace which contains all the information on $f$ is sometimes referred to as \textit{Sufficient dimension reduction} (\cite{Cook1998Regression}). Among the various methods proposed, 
Sliced Inverse Regression (SIR) (\cite{Li1991Sliced}) is one of the most studied. The first article to include the framework of growing ambient dimension $d_n$ shows the consistency of SIR only under $d_n = O(n^{1/2})$ (\cite{Zhu2006Sliced}).
Later, \cite{Lin2018Consistency} show that the phase transition phenomenon occurs for $d_n = o(n)$. In other words, SIR-based estimators are consistent only if $d_n/n \tendto{n} 0$ and this growth rate appears to be optimal (\cite{Lin2021Optimality}).

The difference between growth rates encountered in variable selection and in Sufficient dimension reduction has led recently to the emergence of methods combining both approaches.
If $f$ depends on a $\din$-dimensional subspace $\mathcal{S}_n$ which can be described by linear combinations of only a small number of variables, then we can perform both variable selection and Sufficient dimension reduction over the selected variables. This method, studied for example in \cite{Lin2021Optimality}, \cite{Lin2019Sparse}, \cite{Tan2020SparseSIR}, and \cite{Zeng2024SDR}, allows a return to the exponential growth of $d_n$.

The aim of this article is to perform both function and subspace estimation in the case where no hypothesis is made on $\mathcal{S}_n$ and to derive the maximum growth rate for $d_n$.
Our analysis is done in the nonparametric Bayesian framework introduced by \cite{Ghosal2000Cvrate}. Among the advantages of this approach, the use of very versatile priors, such as Gaussian processes \cite{Vaart2008ContrateGaussian}, allows adaptation to both smoothness and dimension at near minimax rates with a single Bayesian procedure (\cite{Vaart2009AdaptBayes}, \cite{Tokdar2010Bayesian}, \cite{Jiang2021VariableSelection}). This also avoids the complications associated with kernel methods; see for example the introduction of \cite{Shen2013Adaptive}.

In this context, the work of Tokdar, Zhu, and Ghosh \cite{Tokdar2010Bayesian} is one of the first to include a hierarchical prior with a parameter on the subspace. 
They use a uniform prior on the Grassmannian of dimension $\dint = \din$ and a logistic Gaussian process prior for the conditional density function. The authors are able to derive posterior consistency for both the conditional density function and the subspace. 
After which near minimax contraction rates are obtained in \cite{Tokdar2011DimensionAdapt} by extending the framework introduced in \cite{Vaart2009AdaptBayes}.
More recently, \cite{Jiang2021VariableSelection} show that for variable selection, the estimation of both the regression function and the sparsity pattern can be achieved simultaneously at near minimax rates, even when the ambient dimension $d_n$ grows exponentially with the sample size.
This result is consistent with the rate provided in \cite{Yang2015Minimax}, where the variable selection uncertainty term is logarithmic in $d_n$. 

The paper is organized as follows. 
In Section \ref{sec:Pb_f}, we introduce a hierarchical Gaussian process-based prior for both regression and density estimation. This prior is comprised of a dimension parameter for $\dint$, a uniform prior over $\dint$-dimensional linear subspaces of $\R^{d_n}$, a $\dint$-dimensional Gaussian process, and a rescaling parameter to ensure adaptation to smoothness.
Our first result (Theorem \ref{thm:postcons} in Section \ref{sec:Main_thm}) shows that for the estimation of $f$, near minimax contraction rates can be achieved for dimensions $d_n$ growing not faster than $n^{\dint/(2\beta + \dint)}$. 
Apart from specific effects related to the shape of the compact support of $f$ and using the idea that the minimax rate can be decomposed in two terms, one for the estimation of $f$ in intrinsic dimension and one for the model selection uncertainty (\cite{Yang2015Minimax}), we argue that we cannot expect a faster rate in view of the polynomial growth of $d_n$ in Sufficient dimension reduction.
Notably, this rate corresponds to the limiting case where the minimax-error on the estimation of the central subspace $\optsub_n$ is of the same order as the error on $f$.
In Section \ref{sec:sub_recov}, Theorem \ref{thm:subspace_recovery}, we show that for fixed ambient dimension $d_n = d$ and intrinsic dimension $\dint = \din$, the hierarchical Bayes procedure contracts to a subspace that contains $\mathcal{S}_n = \mathcal{S}$ and we conjecture that this subspace is exactly $\mathcal{S}$. 
Our estimation result for $f$ combines the standard arguments used in \cite{Tokdar2011DimensionAdapt} and \cite{Jiang2021VariableSelection} together with new results from \cite{Castillo2024DHGP} for the growing intrinsic dimension framework.
As for the contraction around the central subspace $\mathcal{S}$, we show that an error on the estimation of $\mathcal{S}$ leads to an error on $f$ which contradicts the previously established contraction to $f$.
The proofs of the main results (Theorems \ref{thm:postcons} and \ref{thm:subspace_recovery}) are postponed to Appendices \ref{sec:Proof_MainTh} and \ref{sec:Proof_thmrecov} while Appendix \ref{sec:Lem} is dedicated to useful lemmas.

There exists a vast literature on nonparametric estimation with adaptation to structure in frameworks more general than ours, but it generally focuses on either function estimation or structure estimation but rarely both (at the exception of \cite{Jiang2021VariableSelection}).
Structure estimation includes the aforementioned principal component analysis and Sufficient dimension reduction, whose nonlinear counterpart is manifold learning (\cite{Li2022Spherelets}, \cite{Dunson2024InferManifold}). Without further assumptions, the general problem of inferring a manifold using only the posterior contraction to $f$ as well as the regularization properties offered by the prior appears to be very challenging and to the best of our knowledge, no such work is available yet.
On the other hand, Bayesian procedures for function estimation over a Riemannian manifold (\cite{Yang2016Bayesian}, \cite{Berenfeld2024DensityOnManifold}, \cite{Rosa2023MGPonManifold}), or more general subsets (\cite{Dunson2022GraphGP}, \cite{Tang2024AdaptiveBayes}), even unknown, have proven to be successful and fully adaptive to the intrinsic dimensionality but only when the ambient dimension remains fixed.
Other flexible models such as deep neural networks have been extensively studied and offer similar adaptation properties in Bayesian or non-Bayesian contexts (\cite{Nakada2020AdaptApprox}, \cite{Schmidt2020Nonparametric}, \cite{Castillo2024PostInfDNN}).
The Gaussian counterpart of deep neural networks, deep Gaussian processes (\cite{Finocchio2023DGP}, \cite{Bachoc2025DGP}, \cite{Castillo2024DHGP}) is also known to adapt to the true dimensionality, even in the high-dimensional case with growing $d_n$ and $\dint$ (\cite{Castillo2024DHGP}). However, we show in Section \ref{sec:PostCons} that our results cannot be deduced from the existing literature.

\section{Problem formulation}\label{sec:Pb_f}

\subsection{Notation and definitions}\label{sec:Notation}
	
	We begin with the definition of standard function spaces.
	Let $U$ be a bounded convex subset of $\R^d$, with $d \in \N^*$.
	For $\beta > 0$, write $\beta = k + r$ with $k$ a nonnegative integer and $r \in (0, 1]$.
	The \textit{Hölder space} $\holder^\beta (U)$ is the space of all functions $f: U \to \R$ that are $k$-times differentiable and whose partial derivatives of order $(k_1, \ldots, k_d)$, with $k_1, \ldots, k_d$ nonnegative integers such that $k_1 + \cdots + k_d = k$, are Lipshitz functions of order $r$. That is, there exists a constant $K$ such that
	\[ \abs{\frac{\partial ^k }{\partial_1^{k_1} \cdots \partial_d^{k_d}} (f(x) - f(y))} \ \leq \ K \nor{x - y}{}^r, \]
	for all pairs $x, y \in U^2$ and where $\nor{\cdot}{}$ is the Euclidean norm.
	We endow $\holder^\beta (U)$ with the Hölder sup-norm
	\begin{equation}
		\holdern{f} \ = \ \max \Biggpat{ \max_{|\balpha| \leq k} \nor{\partial^\balpha f}{\infty}, \ \max _{\balpha : |\balpha| = k} \sup_{\substack{x, y \in U^2 \\ x \neq y}} \frac{|\partial^\balpha f(x) - \partial^\balpha f(y)|}{|x - y|_\infty^r}},
	\end{equation}
	where we use the multi-index notation $\partial^\balpha = \partial^{|\balpha|}/\partial_1^{\alpha_1} \cdots \partial_d^{\alpha_d}$ for $\balpha = (\alpha_1, \ldots, \alpha_d) \in \N^d$ and $|\balpha| = |\balpha|_1$. The Hölder ball of radius $K$, denoted $\holder^\beta(U ; K)$, is then the space of all functions $f \in \holder^\beta (U)$ such that $\holdern{f} \leq K$.
	

	We use the following asymptotic notation: if $f$ and $g$ are two real functions over an arbitrary set $S$, then we write $f \lesssim g$ if there exists a universal constant $c$ such that $|f(s)| \leq c \cdot |g(s)|$ for all $s \in S$. The notation $\gtrsim$ is defined in the same way and we write $f \asymp g$ when both $f \lesssim g$ and $f \gtrsim g$ hold.
	
	\medskip
	
	To model the central subspace $\mathcal{S}_n$, we will use the inverse images of isometries instead of the Grassmannian.
	For $d \in \N^*$, we denote by $\O_d$ the space of linear isometries over $\R^d$.
	In addition, the introduction of canonical subspaces and ``component filters'' notation will be very convenient when dealing with sparsity.
	More precisely, for $x \in \R^d$ and $\v \in \{0, 1\}^d$, we denote by $|\v|$ the number of ones in $\v$, by $x_{\v} := (x_j : v_j = 1, \ 1 \leq j \leq d) \in \R^{|\v|}$ the sub-vector with components selected according to $\v$, and for $y \in \R^{|\v|}$, by $y^{\v} := (\tilde{y}_j)_{1 \leq j \leq d}$ the vector in $\R^d$ with $\tilde{y}_j = 0$ if $v_j = 0$ and $\tilde{y}_j = y_i$ if $v_j$ is the $i$-th one in $\v$.
	
	Moreover, for any integer $b \in \eint{1, d}$, we denote by $\b$ the vector $\sum _{i = 1}^b e_i$, where $\{e_i : 1 \leq i \leq d\}$ is the canonical basis of $\R^d$. The dimension $d$ of the ambient space is implicit in this notation. 
	
	Finally, for $\v \in \{0, 1\}^d$, we denote by $E_\v$ the linear span of $\{e_i : v_i = 1 \}$ and by $E_{1 - \v}$ the linear span of $\{e_i : v_i = 0 \}$. Clearly, $E_{1-\v}$ is the orthogonal complement of $E_\v$.

	\medskip

	The proof of Theorem \ref{thm:postcons} involves measuring the complexity of the space where the prior puts its mass. This measure is carried out via \textit{metric entropy}.
	Given a subset $B$ of a metric space $(E, d)$ and a radius $\varepsilon > 0$, we can define the following numbers: 
	\begin{itemize}
		\item the $\varepsilon$\textit{-packing number} $D(\varepsilon, B, d)$ is the maximum number of points in $B$ such that the distance between every pair is at least $\varepsilon$,
		\item the $\varepsilon$\textit{-covering number} $N(\varepsilon, B, d)$ is the minimum number of open balls of radius $\varepsilon$ needed to cover $B$.
	\end{itemize}
	The logarithms of the packing and the covering numbers are called the \textit{entropy} and the \textit{metric entropy} respectively.

\subsection{Bayesian framework for density estimation and regression}

Our main result will be stated for two statistical settings: \textit{density estimation} and \textit{fixed} or \textit{random design regression with Gaussian error}. As we will work with subspaces that are not orthogonal with the canonical axes, the usual support $[0, 1]^d$ for the density or the regression function will be replaced by the unit ball $\U_d := \{x \in \R^d, \nor{x}{} \leq 1 \}$. 
For a given number of observations $n$, the density or the regression function will be characterized by a functional parameter $f_n^* : \U_{d_n} \to \R$. 
The ambient dimension $d_n$ is allowed to grow with $n$ but $f_n^*$ is supposed to depend only on a subspace $\mathcal{S}_n$ with slowly-varying dimension $\dint < d_n$.
A prior on $\dint$ and on the subspace itself will be introduced later to ensure the dimension adaptability.
The prior on the true parameter $f_n^*$ will consist of a projected Gaussian random variable $W_n$ with values in the Banach space $(\mathcal{C}(\U_{d_n}), \nor{\cdot}{\infty})$.
Now let us describe the two statistical settings introduced above.

\paragraph{Density estimation}

Suppose we observe an i.i.d.\ sample $X_1, \ldots, X_n$ from a law $P_n^*$ over $\U_{d_n}$, which admits a continuous density $p_n^*$ relative to the Lebesgue measure on $\R^{d_n}$.
The prior $W_n$ puts its mass on a space that is far too large compared to the space of continuous densities. So to correctly retrieve $p_n^*$, we will work with the parametrized density $p_{n, W_n}$ where, for $w \in \mathcal{C}(\U_{d_n})$, 
\begin{equation}\label{eq:density}
	p_{n, w}(x) \ := \ \frac{e^{w(x)}}{\int_{\U_{d_n}} e^{w(x)} dx}.
\end{equation}
Here the exponential forces the prior to charge only nonnegative functions while the renormalization ensures that $p_{n, w}$ integrates to one. 
The true density $p_n^*$ will then be encoded by the parameter $f_n^* \in \mathcal{C}(\U_{d_n})$ such that $p_n^* = p_{n, f_n^*}$. 
In this way, all the assumptions on the true parameter $f_n^*$ can be transferred to the density $p_n^*$.
That is, $p_n^*$ is supposed to depend only on the $\dint$-dimensional subspace $\mathcal{S}_n$ of $\R^{d_n}$.

The natural metric between two densities $p$ and $p'$ is the Hellinger distance defined by $h(p, p') = \nor{\sqrt{p} - \sqrt{p'}}{2}$, where $\nor{\cdot}{2}$ is the $L^2$-norm with respect to the Lebesgue measure. Consequently, if the parameter space is embedded with a prior $\Pi_n$, we will say that the posterior \textit{contracts} to $p_n^*$ at rate $\suite{\eps}$ if, for any sufficiently large constant $M$, 
\begin{equation}\label{eq:PC_dens}
	\Prob_n^* \cro{\Pi_n \pat{f \in \mathcal{C}(\U_{d_n}) : h(p_{n, f}, p^*_n) > M \eps _n \ | \ X_1, \ldots, X_n}} \tendto{n} 0 , 
\end{equation}
where $\Prob_n^*$ is the joint law of $(X_1, \ldots, X_n)$.

\paragraph{Regression with Gaussian error}

In a regression problem, the covariates can be either predetermined for each observation, this is the \textit{fixed design} case, or can be part of the observation themselves. In the later case, the covariates can be considered as random; this corresponds to the \textit{random design} case. The notion of posterior contraction differs slightly between these two situations and some clarifications are in order.

\subparagraph{Fixed design}

In this setting, we consider a sample of $n$ real observations $Y_1, \ldots, Y_n$ satisfying the model $Y_i = f_n^*(x_i) + \epsilon_i$, with $\epsilon_i \sim \mathcal{N}(0, \sigma^2)$ where the $x_i$'s for $i \in \eint{1, n}$ are $n$ fixed covariates in $\U_{d_n}$ and where the $\epsilon_i$'s are $n$ i.i.d.\ univariate Gaussian random variables with zero mean and standard deviation $\sigma$.
As previously, the regression function $f_n^* : \{x_i : i \in \eint{1, n}\} \to \R$ is supposed to depend only on a $\dint$-dimensional subspace of $\R^{d_n}$.

We will use $W_n$ directly as a prior for the regression function because $W_n$ can be viewed by restriction as a Gaussian process over the space $\mathcal{X}_n := \{x_i : i \in \eint{1, n}\}$ of design points.
To quantify the posterior contraction, we introduce the design dependent semi-metric $\nor{\cdot}{n}$ defined as the $L^2(\Prob_n^x)$-norm for the empirical measure $\Prob_n^x = n^{-1} \sum_{i = 1}^n \delta_{x_i}$ of the design points.
If the space of regression functions over $\mathcal{X}_n$ is embedded with a prior $\Pi_n$, we will say that the posterior \textit{contracts} to $f_n^*$ at rate $\suite{\eps}$ if, for any sufficiently large constant $M$, 
\begin{equation}\label{eq:PC_Fdesign}
	\Prob_n^* \cro{\Pi_n \pat{f \in \mathcal{C}(\mathcal{X}_n) : \nor{f - f_n^*}{n} > M \eps _n \ | \ Y_1, \ldots, Y_n}} \tendto{n} 0 , 
\end{equation}
where $\Prob_n^*$ is the joint law of $(Y_1, \ldots, Y_n)$.

\subparagraph{Random design}

Here, we observe $n$ i.i.d.\ pairs $(X_1, Y_1), \ldots, (X_n, Y_n)$ such that $Y_i = f_n^*(X_i) + \epsilon_i$, with i.i.d.\ $\epsilon_i \sim \mathcal{N}(0, \sigma^2), \ \sigma \in [1, 2]$, and where the $X_i$'s are random covariates over $\U_{d_n}$ independent of the $\epsilon_i$'s and admitting a common density $G_n$ that is bounded away from zero. 
For the sake of simplicity, the standard deviation $\sigma$ is restricted to the interval $[1, 2]$ but these bounds can be relaxed, see Remark \ref{sec:RDregress} for details.
Again, the regression function $f_n^* : \U_{d_n} \to \R$ is supposed to depend only on a $\dint$-dimensional subspace of $\R^{d_n}$.
Moreover, we use $W_n$ directly as a prior for the regression function.
The natural metric for this problem is the $L^2(G_n)$-norm denoted by $\nor{\cdot}{2, G_n}$ where $G_n$ is identified with the law of one covariate. 
This metric is not equivalent to the Hellinger metric, which is used in the proof of Theorem \ref{thm:postcons}, unless all regression functions are uniformly bounded by a constant $Q > 0$. This condition can be fulfilled by projecting the prior on the space of all functions uniformly bounded by $Q$, as proposed in \cite{Gine2011Rates}. However, this would force us to rewrite the proof of Theorem \ref{thm:postcons} only for this setting. Instead, we directly post-process the posterior to integrate this constraint as in \cite{Yang2016Bayesian}.
Then, the formulation of posterior consistency becomes as follows.
Considering a prior $\Pi_n$ over the regression functions, we will say that the posterior \textit{contracts} to $f_n^*$ at rate $\suite{\eps}$ if, for $Q > 0$ and any sufficiently large constant $M$, 
\begin{equation}\label{eq:PC_Rdesign}
	\Prob_n^* \cro{\Pi_n \pat{f \in \mathcal{C}(\U_{d_n}) : \bignor{f^Q - {f_n^*}^Q}{2, G_n} > M \eps _n \ | \ (X_1, Y_1), \ldots, (X_n, Y_n)}} \tendto{n} 0 , 
\end{equation}
where $\Prob_n^*$ is the joint law of $(X_1, Y_1), \ldots, (X_n, Y_n)$ and $f^Q := (f \vee -Q) \wedge Q$ is the truncated version of $f$.

\bigskip

\section{Main result for the functional parameter}\label{sec:Main_thm}

Our sparsity framework takes the form of a hierarchical model with three components: an intrinsic dimension $\dint$ less than or equal to the ambient dimension $d_n$, an intrinsic functional parameter $\fint$ over a $\dint$-dimensional space, and a rotation $q_n^*$ which maps a standard subspace of $\R^{d_n}$ to $\optsub_n$. The true parameter $f_n^*$ is a sparse continuation of $\fint$ whose support has been rotated to match $\optsub_n$.

\begin{assumption}[Sparsity of the true parameter]\label{ass:spars}
For all $n \in \N$, there exist an intrinsic dimension $\dint \in \N$, an intrinsic parameter $\fint \in \mathcal{C}(\U_\dint)$, and a linear isometry $q_n^* \in \O_{d_n}$ such that $d_n \geq \dint$, and $f_n^*(x) = \fint\pat{(q_n^*(x))_{\bdint}}$, for all $x \in \U_{d_n}$.
\end{assumption}

The intrinsic functional parameter $\fint$ is sometimes called the \textit{core function}. The use of isometries instead of vector subspaces permits us to avoid the manipulation of the Grassmannian, replaced by the more convenient orthogonal group $\O_{d_n}$. This will considerably simplify the proof of Theorem \ref{thm:postcons}.
Note that the central subspace $\optsub_n$ is now equal to $(q_n^*)^{-1}(E_{\bdint})$.
The next property is straightforward.

\begin{proper}\label{prop:Parci}
For $n \in \N$, $f_n^*$ is constant on the intersection between $\U_{d_n}$ and the affine subspaces $\optsub_n^\perp + x$, for $x \in \R^{d_n}$, where $\optsub_n^\perp$ is the orthogonal complement of $\optsub_n$.
\end{proper}


In addition to dimension adaptation, the present setting allows the core function $\fint$ to be arbitrarily smooth (in an Hölder sense) while maintaining near-minimax contraction rates. As in \cite{Castillo2024DHGP}, a control on the Hölder-norm, which will be clarified later, is needed.

\begin{assumption}[Smoothness of $\fint$]\label{ass:smooth}
There exist $\beta > 0$ and $K_n \geq 1$ such that $\fint \in \holder^\beta(\U_{\dint} ; K_n)$ with $K_n$ growing at most polynomially with $n$.
\end{assumption}

\subsection{Prior specification}

Similarly to \cite{Tokdar2011DimensionAdapt}, we embed the parameter space with a hierarchical prior that consists of a random sparsity pattern $(b, q)$, providing the intrinsic dimension and the orientation, a rescaling parameter $a$, and a core function modeled by a standard squared exponential Gaussian process. The prior on the sparsity pattern allows adaptation to intrinsic dimension, whereas a carefully chosen rescaling parameter together with the above infinitely smooth process has been shown to perform smoothness adaptation (\cite{Vaart2009AdaptBayes}).

For $n > 0$, let $W = (W(x) : x \in \R^{d_n})$ be a standard squared exponential Gaussian process on $\R^{d_n}$; that is, a centered Gaussian process with covariance kernel
\[ \Esp [W(s)W(t)] \ = \ \exp \bigpat{-\nor{s - t}{}^2}, \qquad \text{for all } s, t \in \R^{d_n}, \]
where $\nor{\cdot}{}$ is the Euclidean norm.

Let $a > 0$, $b \in \eint{1, d_n}$, and $q \in \O_{d_n}$. We define $W_x^{a, b, q} := W(a \diag(\b) \cdot q(x))$ and $W^{a, b, q} := (W_x^{a, b, q} : x \in \R^{d_n})$ a rescaled Gaussian process with sparsity pattern $(b, q)$, where $\diag(\b)$ is the diagonal matrix with diagonal vector $\b$.
Then, the process $W^{a, b, q}$ is constant on affine subspaces $q^{-1} (E_{1-\b}) + x$, for $x \in \R^{d_n}$ (as in Property \ref{prop:Parci}) and if $R := q^{-1}\diag(\b)q$ is the orthogonal projection onto $q^{-1}(E_\b)$, then $W^{a, b, q}_x = W^{a, b, q}_{Rx}$, for all $x \in \R^{d_n}$.

To work properly with $W^{a, b, q}$, we have to verify that its law identifies with the law of a $b$-dimensional standard squared exponential Gaussian process. To do so, define
\begin{align*}
\Upsilon : \ \R^b & \ \to \ q^{-1}(E_\b) \\
x & \ \mapsto \ \frac1a q^{-1} (x^\b),
\end{align*}
a bijection with inverse $\Upsilon^{-1} (t) = a (qt)_\b$ for $t \in q^{-1}(E_\b)$. Then, $W^{a, b, q}_{\Upsilon(x)} = W(x^\b)$ for all $x \in \R^b$. 
Let us introduce $\Wtilde := (W^{a, b, q}_{\Upsilon(x)}, x \in \R^b)$.
For all $(x, y) \in \R^b \times \R^b$, we have
\[ \Esp [\Wtilde(x) \Wtilde(y)] \ = \ \Esp [ W(x^\b) W(y^\b)] \ = \ e^{-\nor{x^\b - y^\b}{}^2}. \]
So $\Wtilde$ is a standard squared exponential Gaussian process in dimension $b$ that does not depend on either $a$ nor $q$. Moreover, $W^{a, b, q}_t = \Wtilde(\Upsilon^{-1}(Rt))$.

From now on, $W^{a, b, q}$ will refer to the restriction on $\U_{d_n}$ of this process. Then, the hierarchical prior on the parameter $f \in \mathcal{C}(\U_{d_n})$ with stochastic subspace selection is defined as the joint law $\Pi_n$ of $(\Gamma, \Theta, W^{A, \Gamma, \Theta})$, where $A$ is the scaling parameter, $\Gamma \in \eint{1, d_n}$ is the prior on the subspace dimension, and $\Theta$ is the prior on the orientation. All $A, \Gamma$, and $\Theta$ are supposed to be independent of $W$.
\begin{assumption}[Prior on intrinsic dimension]
	The prior $\Gamma$ is defined by a collection of probability vectors $(\pi_{\Gamma, n}(b) \ : \ 1 \leq b \leq \dintup)$ such that $\pi_{\Gamma, n}(\dint) \geq 1/n$ 
	with $\dintup$ given by Assumption \ref{ass:growth}.
\end{assumption}
The last part of the assumption is not very restrictive in view of Assumption \ref{ass:growth}. In fact, $\pi_{\Gamma, n}$ can be taken as a uniform distribution on $\eint{1, \dintup}$ or a truncated Poisson distribution with the restriction $\dintup = \cint \sqrt{\log n}$.

When establishing the prior mass condition \eqref{eq:priormasscond}, we integrate over a set of isometries whose prior mass must be lower bounded. This lower bound is derived in Lemma \ref{lem:measure1} under the condition that the law of the stochastic isometry $\Theta$ is translation invariant. That is, for all subsets $\mathcal{Q}$ of $\O_{d_n}$ and for all $q \in \O_{d_n}$, we need $\Prob (\Theta \in q \cdot \mathcal{Q}) = \Prob(\Theta \in \mathcal{Q})$. Therefore, the law of $\Theta$ is taken as the unit Haar measure on $\O_{d_n}$, the only probability measure that is translation invariant on $\O_{d_n}$.

As for the scaling parameter $A$, there exists a collection of probability measures $\pi_{n, d}$ over $(0, \infty)$, $0 \leq d \leq d_n$, satisfying $A \ | \ (\Gamma = d) \sim \pi_{n, d}$. 
For convenience, the notation $\pi_{n, d}$ will refer to a probability measure as well as its density.
\begin{assumption}[Rescaling measures]\label{ass:resc}
There exist constants $D_1, D_2, D_3$, and $D_4$ such that for all $n \in \N^*$ and $d < d_n$, the density $\pi_{n, d}$ satisfies
\begin{enumerate}
\item for all sufficiently large $a$, $\pi_{n, d}(a) \geq D_1 e^{-D_2 a^d (\log a)^{d+1}}$ ;
\item for all $a \geq 0$, $\pi_{n, d}(a) \leq D_3 e^{-D_4 a^d (\log a)^{d+1}}$.
\end{enumerate}
\end{assumption}
Assumptions similar to Assumption \ref{ass:resc} are standard, see for instance Equation (3.4) in \cite{Vaart2009AdaptBayes} or Assumption 5 in \cite{Jiang2021VariableSelection}.
For example, this assumption is satisfied if, for all $n \in \N^*$ and $d < d_n$, $A^d (\log A)^{d+1} \ | \ (\Gamma = d)$ is an exponential law with parameter independent of $d$ and $n$  (indeed, if $g(A)$ has density function $f$, with $g$ differentiable and strictly increasing, then $A$ has density function $(f \circ g) \cdot g'$).

The next section gives some precision about the reproducing kernel Hilbert space (RKHS) of $W^{a, b, q}$. The content is a bit technical and can be skipped at first reading.

\subsection{Reproducing kernel Hilbert space of $W^{a, b, q}$}

One of the advantages of choosing a Gaussian process prior is that the contraction rate depends explicitly on the small ball probability and on the relative position of the parameter with respect to the RKHS associated with the process.
This section is dedicated to the basic properties of this space.
For elementary definitions and further precision about the link between the contraction rate and the RKHS, we refer the reader to \cite{Vaart2008ContrateGaussian} and \cite{Vaart2008RKHS}.

\begin{notation}
	We denote by $\mathcal{C}(\U_d \ | \ q^{-1}(E_\b))$ the space of continuous functions on $\U_d$ which are constant on affine subspaces $q^{-1} (E_{1-\b}) + x$, for $x \in \U_d$.
\end{notation}

We introduce the operator
\[ \Lambda_{b,q} : 
\begin{cases} \mathcal{C}(\U_b) &\kern-6pt \to \ \mathcal{C}(\U_d \ | \ q^{-1}(E_\b)) \\[0.3cm]
	\phantom{\L^2} f &\kern-6pt \mapsto \ \Lambda_{b,q} \, f : 
	\begin{cases} 
		\U_d &\kern-6pt \to \ \R \\
		\ x &\kern-6pt \mapsto \ f((qx)_\b),
	\end{cases}
\end{cases}
\]
or simply $\Lambda_q$ or $\Lambda$ when $b$ and $q$ are implicit, so that $W^{a, b, q} = \Lambda_{b,q} (\Wtilde^a)$, where $\Wtilde^a = (\Wtilde_{at}, t \in \U_b)$ is the process $\Wtilde$ introduced above rescaled by $a$ and restricted to $\U_b$.
It is a bijective linear map and also an isometry if the domain and the codomain are endowed with the uniform norm. In particular, the map $\Lambda$ is continuous. According to Lemma 7.1 in \cite{Vaart2008RKHS}, if $\widetilde{\H}_a$ is the RKHS of $\Wtilde^a$, then the RKHS $\H_{a, b, q}$ of $W^{a, b, q}$ is equal to $\Lambda (\widetilde{\H}_a)$.
Let us detail its elements.
The stochastic process RKHS of $\Wtilde^a$ (as defined in \cite{Vaart2008RKHS}) is composed of functions $h : \U_b \to \R$ for which there exists $\psi \in L^2_\C (\mu^{se}_{a, b})$ such that
\begin{equation*}
	h(t) \ = \ \re \int _{\R^b} e^{-i\lambda \cdot t} \psi (\lambda) d\mu^{se}_{a, b} (\lambda) , \quad t \in \U_b,
\end{equation*}
where $\mu^{se}_{a, b}$ is the spectral measure of the $a$-rescaled squared exponential Gaussian process in dimension $b$ with spectral density $f^{se}_{a, b} : t \mapsto (2a\sqrt{\pi})^{-b} \exp(-\frac14 \| t/a \|^2)$ (see Lemma 4.1 in \cite{Vaart2009AdaptBayes}, and the following discussion). We can view $\Wtilde^a$ as a random Gaussian element with values in the Banach space $(\mathcal{C}(\U_b), \nor{\cdot}{\infty})$. Thus, according to Theorem 2.1 in \cite{Vaart2008RKHS}, the stochastic process RKHS and the Banach space RKHS coincide and we can apply Lemma 7.1 from the same reference.
The space $\H_{a, b, q} = \Lambda (\widetilde{\H}_a)$ is then the set of functions
\begin{equation}\label{eq:repr_thm}
	\overline{h} : x \in \U_d \ \mapsto \ \re \int _{\R^b} e^{-i \inner{\lambda}{(qx)_\b}} \psi (\lambda) d \mu^{se}_{a, b} (\lambda), 
\end{equation}
where $\psi$ runs through $L^2_\C(\mu^{se}_{a, b})$ and the RKHS norm is $\nor{\overline{h}}{\H_{a, b, q}} = \nor{\psi}{L^2(\mu^{se}_{a, b})}$.

We remark that functions of the RKHS of $W^{a, b, q}$ have the same sparsity-pattern as the trajectories of $W^{a, b, q}$.
\begin{remark}
	Functions $\overline{h} \in \H_{a, b, q}$ are constant on affine subspaces $q^{-1}(E_{1-\b}) + x$ for $x \in \U_d$.
\end{remark}

As mentioned at the beginning of this section, contraction rates under Gaussian process prior depend on two ingredients: the small ball probability and the relative position of the parameter with respect to the RKHS. 
For a parameter $f \in \mathcal{C}(\U_d \ | \ q^{-1}(E_\b))$ and $\eps > 0$, these two quantities define the \textit{concentration function} $\phi_f^{a, b, q}$, with
\begin{equation}\label{rem:concfunct}
	\phi_f^{a, b, q} (\eps) \ := \ \inf _{h \in \H_{a, b, q} : \nor{h - f}{\infty} < \eps} \nor{h}{\H_{a, b, q}}^2 - \log \Prob \pat{\nor{W^{a, b, q}}{\infty} < \eps}.
\end{equation}

\subsection{Posterior consistency}\label{sec:PostCons}

Define the contraction rates, 
\begin{equation}
	\epsu_n \ := \ n^{-\tfrac{\beta}{2\beta + \dint}} \qquad \text{and} \quad 
	\eps_n \ := \ C_\eps(\beta)^\dint K_n^2 \cdot \epsu_n \cdot (\log n)^{\kappa},
\end{equation}
with $\kappa = \kappa_n := \beta \pat{\dint +1}/(\dint + 2\beta)$, 
where $C_\eps (\beta) = C_\eps$ is a large constant that depends only on $\beta$ and where $K_n$ is an upper bound on the Hölder-norm of $\fint$ (Assumption \ref{ass:smooth}).
By taking $K_n^2 \leq C_\eps^\dint$, the rate $\eps_n$ is equal to the minimax rate times a lower order term of at most $C_\eps^{2\dint} (\log n)^\kappa$. As noted in \cite{Castillo2024DHGP}, a growth of the order $C^\dint$ for the radius of the Hölder ball in dimension $\dint$ is typical, see the discussion below Corollary 3 and Appendix D in the same reference. 

Before we state the theorem, we need a last assumption which determines how the ambient dimension $d_n$ and the intrinsic dimension $\dint$ are allowed to grow with the sample size $n$.
In the spirit of \cite{Yang2015Minimax}, this growth rate must take into account the uncertainty of the subspace estimation. In particular, if we want to keep a near minimax contraction rate, this uncertainty must be controlled by $\epsu_n$.
In Sufficient dimension reduction, minimax rates of SIR estimators are known to be of the order $\sqrt{\dint d_n /n}$ for the projection loss $\nor{P_{\optsub_n} - P_{\optsub_n'}}{F}$ (see \cite{Tan2020SparseSIR}, \cite{Lin2021Optimality} and \cite{Zeng2024SDR}) and similar results are known in principal component analysis (\cite{Cai2013SparsePCA}, \cite{Johnstone2004sparsePCA}).
Ignoring the effects caused by the ellipsoidal support, an error in the estimation of $\optsub_n$ with the projection loss results in an error of the same order in the estimation of the main parameter $f^*_n$ or $p^*_n$, provided that $f^*_n$ or $p^*_n$ fluctuates sufficiently and up to multiplicative constants depending on $\dint$.
This behavior can be justified by early use of elements from the next section (Section \ref{sec:sub_recov}).

For the sake of brevity, we only focus on density estimation, as the arguments also apply for the regression problem. Since we are interested in a minimax bound, we can restrict ourselves to cases where $p^*_n$ is strictly monotone.
In the proof of Theorem \ref{thm:subspace_recovery}, we show that if $\nor{P_{\optsub_n} - P_{\optsub_n'}}{F} > \vartheta$ for some small $\vartheta > 0$ (as in the equation \eqref{eq:lemFrobnorm} from Lemma \ref{lem:dist}), this results in an Hellinger error on $p^*_n$ of the order ${(L/\sqrt{d_n})^{(d_n +2)/2} \cdot \vartheta/\sqrt{\dint}}$, for a window size $L < 1$, see for example the inequality \eqref{eq:Hellingerrd}. Because $p^*_n$ is supposed strictly monotone, the default of constancy is spread out over the support and we can integrate the error not only over a small hypercube $R$ but over a region whose mass does not decrease with $d_n$, thus eliminating the factor $(L/\sqrt{d_n})^{(d_n +2)/2}$ in the preceding error.

Consequently, the maximum growth rate one can expect is constrained by
\begin{equation*}
	\sqrt{\frac{d_n}{n}} \ \leq \ n^{-\tfrac{\beta}{2\beta + \dint}},
\end{equation*}
which leads to $d_n \leq n^{\tfrac{\dint}{2\beta + \dint}}$.
The growth rate of the intrinsic dimension $\dint$ is not affected by the ambient dimension so it can be set similarly to \cite{Castillo2024DHGP}.

\begin{assumption}[Growth of $d_n$ and $\dint$]\label{ass:growth}
	For a given constant $\cint$, the intrinsic dimension $\dint$ and the ambient dimension $d_n$ satisfy,
	\begin{gather}
		\dint \ \leq \ \dintup \ := \ \cint \sqrt{\log n}, \\
		\label{eq:growthd_n} d_n \ \lesssim \ K_n^{4} \cdot n^{\tfrac{\dint}{2\beta + \dint}} \cdot (\log n)^{2\kappa - 2}.
	\end{gather}
\end{assumption}

%
%

\begin{theorem}\label{thm:postcons}
If the parameter space is embedded with the prior $\Pi_n$ and under Assumptions \ref{ass:spars}-\ref{ass:growth}, the posterior contracts at rate $\suit{\eps_n}$ for density estimation (as defined in \eqref{eq:PC_dens}) as well as for regression with fixed or random design (as defined in \eqref{eq:PC_Fdesign} and \eqref{eq:PC_Rdesign}).
\end{theorem}

An examination of $\eps_n$ shows that the contraction rate is improved as the smoothness $\beta$ of $\fint$ grows, unlike $d_n$. This highlights a trade-off between the contraction rate and the growth of the design dimension: fast contractions imply slowly increasing dimension and conversely. 
Moreover, we choose to define $\eps_n$ first and then to constrain $d_n$ according to it, but the converse is straightforward. Given a growth rate larger than \eqref{eq:growthd_n} up to the order $n$, it suffices to derive the corresponding value of $\beta$ and let the posterior contract at the rate corresponding to the underestimated regularity.

In view of Assumption \ref{ass:spars}, it is appealing to model the true parameter with a deep Gaussian process of only two layers: the first modeling the action of the isometry and the second, the core function $\fint$. In fact, minimax contraction rates for deep Gaussian process priors (up to taking tempered posteriors), have recently been derived in the high dimensional setting, where the input dimensionality of the first layer can grow polynomially with $n$ (\cite{Castillo2024DHGP}). 
Unfortunately, in their present form, these results do not lead directly to adaptation to the intrinsic dimension because the first layer (modeling the isometry), from $\U_{d_n}$ to $\U_\dint$, can have dependencies with all input variables.
Then, Theorem 4 in \cite{Castillo2024DHGP} yields contraction rates not faster than $n^{-\beta/(2\beta + d_n)}$.
Still, another point from this reference may produce interesting results. 
It is shown that rescaling priors that allow vanishing lengthscales---such as exponential or Horseshoe priors---can achieve variable selection by freezing the paths in corresponding directions.
In the introduction, we show that variable selection, by dramatically reducing the complexity of the model, can be performed with ambient dimension growing exponentially with $n$, in contrast with the polynomial growth associated with subspace selection. We suggest that using Horseshoes priors with an ambient dimension growing exponentially with $n$, but with only a polynomial number of variables allocated for the subspace, could still lead to posterior contraction at minimax rates without adding another layer to the hierarchical model.

The proof of Theorem \ref{thm:postcons}, postponed in Appendix \ref{sec:Proof_MainTh}, combines the arguments of \cite{Tokdar2011DimensionAdapt} and \cite{Jiang2021VariableSelection}, and uses new results from \cite{Castillo2024DHGP} for the growing intrinsic dimension framework.
The novelty consists in integrating the orientation prior on larger subsets than \cite{Tokdar2011DimensionAdapt} in the prior mass condition to ensure that the growth rate $d_n$ can be attained. We also express the sieve as a finite union indexed by a net in $\O_{d_n}$, where the trivial net used in \cite{Tokdar2011DimensionAdapt} is replaced by a smaller one, with controlled cardinality.

\medskip

\section{Subspace recovery}\label{sec:sub_recov}

In this section, we propose to recover partially the central subspace for density estimation and regression with random design by using the previously established posterior contraction. We will be able to eliminate wrong directions for the orthogonal complement of the central subspace $\optsub_n^\perp$ and consequently to recover a subspace that contains $\optsub_n$. However, posterior contraction does not help to distinguish $\optsub_n$ from bigger subspaces. This problem, which we do not address here, is left to regularization offered by the prior.

To avoid identifiability issues caused by the ellipsoidal support, we suppose that the ambient and intrinsic dimensions $d_n$ and $\dint$ do not depend on $n$. Hence, we drop down the index $n$ and write $d_n = d$, $\dint = \din$, with $d \geq \din$. Without loss of generality, we also suppose that the central subspace $\optsub_n = \optsub := (q^*)^{-1} (E_\dO)$ and the true parameters $\fin$ and $f^*$ (Assumption \ref{ass:spars}) remain fixed.
Theses assumptions are justified by the following considerations. 
For density estimation, if the ambient dimension grows with $n$, the Hellinger metric relative to the Lebesgue measure on $\U_{d_n}$ tends to give more importance to the center of the support, as $n$ tends to infinity. For example, consider a parameter $\fin : \U_2 \to \R$ in dimension two that is everywhere constant except in a small region on the border of $\U_2$, and such that the central subspace $\optsub_n$ is of dimension two. The importance of this small region in the support $\U_{d_n}$, in the Hellinger sense, decreases exponentially with $n$, much faster than the estimation of the true parameter $f_n^*$ in Theorem \ref{thm:postcons}. 
Consequently, for sufficiently large $n$, a constant function $\fin : [0, 1] \to \R$ together with some one-dimensional subspace $\mathcal{S}'$ characterize a density that is in the Hellinger ball of radius $\eps_n$ centered on $f_n^*$; we have no hope of recovering the true subspace by simply using the posterior consistency. 
A similar heuristic can be derived for regression with random design when the joint density of the covariates is uniform or Gaussian.

When measuring discrepancies between subspaces, a number of possibilities come to mind; we will examine three of them. 
Let $\optsub'$ and $\optsub''$ be two subspaces of $\R^d$ of same dimension.
The first loss is directly derived from the isometry prior $\Theta$ and expresses the distance between $\optsub'$ and $\optsub''$ as the minimal distance between isometries that orient $\optsub'$ and $\optsub''$ in operator norm.
Let $\Q_{\optsub'}$ be the set of isometries that send the subspace $\optsub'$ to $E_{\dim \optsub'} := \vect (e_1, \ldots, e_{\dim \optsub'})$: 
\[ \Q_{\optsub'} \ := \ \{ q \in \O_d : q^{-1} (E_{\dim \optsub'}) = \optsub' \}. \]
The distance between $\optsub'$ and $\optsub''$ is then defined as
\begin{equation}
	d_1 \pat{\optsub', \optsub''} \ := \ \min _{q \in \Q_{\optsub'}} \vertiii{q - q'}, \qquad \text{for any } q' \in \Q_{\optsub''},
\end{equation}
where $\vertiii{\cdot}$ is the operator norm with respect to the Euclidean distance in $\R^d$.
Note that $\Q_\optsub$ is the set of optimal isometries and that we want to show that the posterior of $\Theta$ converges to this set.

The second loss is a more natural way to express distances between subspaces. It will be taken as the Hausdorff distance between the intersection of the two subspaces with the unit sphere $\sphere_d$.
Let us define the \textit{directional Hausdorff distance} between two closed and bounded subsets $\mathcal{X}$ and $\mathcal{Y}$ of $\R^d$ as, 
\begin{equation*}
	\dirhaus(\mathcal{X}, \mathcal{Y}) \ = \ \sup _{x \in \mathcal{X}} d(x, \mathcal{Y}),
\end{equation*}
where $d(x, \mathcal{Y}) = \inf _{y \in \mathcal{Y}} \nor{x - y}{}$ is the Euclidean distance between $x$ and $\mathcal{Y}$.
The Hausdorff distance $\dhaus(\mathcal{X}, \mathcal{Y})$ is then the maximum between $\dirhaus(\mathcal{X}, \mathcal{Y})$ and $\dirhaus(\mathcal{Y}, \mathcal{X})$. 
Fortunately, if $\mathcal{X}$ and $\mathcal{Y}$ are the restriction to $\sphere_d$ of linear subspaces of same dimension, say $\optsub'$ and $\optsub''$, the last two quantities are equal. In fact, the directional Hausdorff distance is invariant under isometric transformation, so it suffices to apply an isometry that exchanges $\optsub'$ and $\optsub''$ to switch from one to another.
We thus define:
\begin{equation}\label{def:disthaus}
d_2(\optsub', \optsub'') \ := \ \dirhaus(\optsub' \cap \sphere_d, \optsub'' \cap \sphere_d). 
\end{equation}

The last is an adaptation of the popular \textit{projection loss} in principal component analysis and Sufficient dimension reduction (\cite{Tan2020SparseSIR}). If $P_{\optsub'}$ and $P_{\optsub''}$ are the orthogonal projections on $\optsub'$ and $\optsub''$ respectively, we define:
\begin{equation}\label{eq:Projloss}
	d_3(\optsub', \optsub'') \ := \ \nor{ P_{\optsub'} - P_{\optsub''}}{F},
\end{equation}
where $\nor{\cdot}{F}$ is the Frobenius norm.
It can be seen that the first two distances $d_1$ and $d_2$ are both equal to the Euclidean distance between the principal vectors corresponding to the largest principal angle between $\optsub'$ and $\optsub''$ (see \cite{Bhatia1997MatrixAnalysis}, Exercise VII.1.12, p 202. for a definition of the principal angles). Moreover, $d_3^2(\optsub', \optsub'')$ corresponds to twice the sum of the squared sines of the principal angles; this standard result is derived in the proof of Lemma \ref{lem:dist} for convenience.

As mentioned at the beginning of this section, the posterior consistency will only allow us to recover a subspace of $\R^d$ containing $\optsub$.
We thus extend the three preceding losses to subspaces with different dimensions in such a way that $d_i(\optsub, \optsub') = 0$ as soon as $\optsub \subset \optsub'$.
Let $\optsub'$ be a subspace of $\R^d$ such that $\dim \optsub' > \din$. The general losses are defined as follows:
\begin{equation}\label{eq:subspaceminmax}
	d_i(\optsub, \optsub') \ := \ \min _{\substack{\overline{\optsub} \supset \optsub \\ \dim \overline{\optsub} = \dim \optsub'}} d_i(\overline{\optsub}, \optsub') \ = \ \min _{\substack{\underline{\optsub} \subset \optsub' \\ \dim \underline{\optsub} = \din}} d_i(\optsub, \underline{\optsub}), \qquad i = 1, 2, 3,
\end{equation}
where the first minimum is taken over all subspaces of dimension $\dim \optsub'$ containing $\optsub$ and the second minimum is taken over $\din$-dimensional subspaces of $\optsub'$. A proof of the equality is provided in Appendix \ref{sec:Lem}.

A crucial assumption to eliminate the subspaces of dimension smaller than $\din$ and the subspaces that do not contain $\mathcal{S}$ is to suppose that the core function is non-constant in all directions. More precisely, the lack of constancy for each direction has to be detectable in Hellinger or in $L^2$ distances. The formulation of this hypothesis differs depending on whether we estimate a density or a regression function.
For the density estimation problem, recall that the true density $p^* = p_{f^*}$ is characterized by $f^*$ via the transformation \eqref{eq:density}. Moreover, $p^*$ can be viewed as the sparse continuation of a function $p_0$ over $\U_\din$, equal to $p_\fin$ up to a renormalization factor that depends on $d$. Note that $p_0$ is not necessarily a density on $\U_\din$ so the notation $h(f, g)$ will designate from now on the $L^2$-distance between the square roots of $f$ and $g$ as soon as the last quantity exists.



\begin{assumption}\label{ass:directdetect}
	There exist a constant $D$ and a window size $L < 1$ such that for all unit vector $\v$, there exists $o \in  B_{1-L}(0)$ such that for all $0 < l \leq L$, for all $t \in B_{L/2}(o)$, and for all constant $c \in \R$, 
	\begin{align*}
	 h^2 ({p_0}_{|I} ; \abs{c}) \ & \geq \ D \cdot l^3, \qquad \text{(density estimation)}, \\
	\text{or} \quad \bignor{\fin^Q_{|I} - c}{2}^2 \ & \geq \ D \cdot l^3, \qquad \text{(regression with random design)},
	\end{align*}
	where $I := \ ]t - \frac{l}{2} \v ; t + \frac{l}{2} \v[$ and where $\fin^Q = (\fin \vee -Q) \wedge Q$ as in \eqref{eq:PC_Rdesign}.
\end{assumption}

Assumption \ref{ass:directdetect} seems a bit technical at first glance but it can be shown that it is satisfied as soon as $\fin^Q$ and $p_0$ are continuously differentiable over $\U_\din$ with $\din$ points such that the gradients at these points are linearly independent. The precise statement can be found in Lemma \ref{lem:gradient}.

\begin{theorem}\label{thm:subspace_recovery}
	Under Assumption \ref{ass:directdetect} and the assumptions of Theorem \ref{thm:postcons}, we have, for some rates $(\delta^{(i)}_n)_n$ tending to zero and for both density estimation and regression with random design problems,  
	\begin{gather}
		\label{eq:recov_infd0} \Pi_n \pat{\Gamma < \din \ | \ Z_1, \ldots, Z_n} \tendto{n} 0, \qquad \text{in } \Prob_n^*\text{-probability},\\
		\label{eq:recov_supd0} \Pi_n \pat{\Gamma \geq \din \text{ and } d_i(\optsub, \Theta^{-1} (E_{\boldsymbol{\Gamma}})) \geq \delta^{(i)}_n \ | \ Z_1, \ldots, Z_n}\tendto{n} 0,  \qquad \text{in } \Prob_n^*\text{-probability},
	\end{gather}
	for $i = 1,2,3$ and where $Z_j = X_j$ for density estimation and $Z_j = (X_j, Y_j)$ for regression with random design, $j \in \eint{1, n}$.
	Moreover, for density estimation, we can take:
	\[\delta^{(1)}_n \ = \ \delta^{(2)}_n \ = \ \sqrt{2/D} \pat{d/L^2}^{(d+2)/4}\eps_n, \qquad \delta^{(3)}_n \ = \ \sqrt{2\din} \cdot \delta^{(1)}_n, \]
	provided $n$ is sufficiently large such that these quantities are strictly less than 1;
	for regression with random design, we can take the same rates divided by the square root of the minimum of the design density $G_n = G$ over its support.
\end{theorem}

Theorem \ref{thm:subspace_recovery} ensures that the central subspace $\mathcal{S}$ can be recovered as soon as the intrinsic dimension $\din$ is known. Subspaces of dimension smaller than $\din$ are also eliminated but the theorem does not reject those of dimension greater than $\din$. We conjecture that the prior mass on those spaces tends to vanish, for reasons similar to those exposed in \cite{Jiang2021VariableSelection}. Indeed, introducing a penalization on larger dimensions if necessary, it should be possible to show that the posterior cannot contract as fast as the minimax rate for $\din$ if a subspace of greater dimension is chosen.
This regularization property of the prior is well expected and has already been observed in other contexts (\cite{Rousseau2011Asymptotic}).
As discussed in the introduction of this section, the estimation of the central subspace is made under the assumption that $d$ is fixed with $n$ mainly because of the identifiability issue caused by the ellipsoidal support. This translates into an extremely penalized rate $\delta_n$ in the main theorem, with the factor $(d/L^2)^{(d+2)/4}$ coming from the cubic integration region constrained to fit within the unit ball. We believe that this restriction can be relaxed by extending the support $\U_{d}$ to the full ambient space $\R^{d}$, as in \cite{Jiang2021VariableSelection}. In this case, the hypercube over which we integrate the Hellinger distance in the proof of Theorem \ref{thm:subspace_recovery} can be taken as the product space of a square of side $L$ in directions $\boldsymbol{\Delta}$ and $\boldsymbol{\Lambda}$ times $\R^{d-2}$. Then, the integrated error should no longer depend on $d$ and consistency to the true subspace should follow. Further investigations in this direction might be worthwhile.
Another possibility, that keeps the compact support, could be to strengthen Assumption \ref{ass:directdetect} to force the default of constancy of $\fin$ to be detectable in Hellinger or $L^2(G_n)$ metrics regardless of the ambient dimension $d_n$. Instead of our local hypothesis, this would force the default of constancy to be more spread out over the support. Some popular models such as strictly convex or monotone functions fulfill these restrictions. 

The proof of Theorem \ref{thm:subspace_recovery} is postponed to Appendix \ref{sec:Proof_thmrecov}.

\paragraph{Acknowledgments}

We acknowledge the support of the French Agence Nationale de la Recherche (ANR) under reference ANR-21-CE40-0007 (GAP Project).
We are grateful to an Associate Editor and three Referees for their comments that led to an improvement of the paper.
We are also grateful to Iain Henderson for his help in improving the presentation of the manuscript.

\section{Appendix}

\subsection{Proof of Theorem \ref{thm:postcons}}\label{sec:Proof_MainTh}

The proof of Theorem \ref{thm:postcons} is based on Theorem 2.1 in \cite{Ghosal2000Cvrate}. The general outline is a combination of the arguments of \cite{Tokdar2011DimensionAdapt} (itself derived from \cite{Vaart2009AdaptBayes}) and \cite{Jiang2021VariableSelection}. For the growing intrinsic dimension framework, we use the new dimension-dependent lemmas from \cite{Castillo2024DHGP}.
Concretely, it suffices to show that there exists a sequence of sets $\B_n \subset \mathcal{C}(\U_{d_n})$ (referred to as a \textit{sieve}), such that the following three conditions hold for all sufficiently large $n$: 
\begin{gather}
	\label{eq:priormasscond} \Pi_n \pat{\nor{W^{A, \Gamma, \Theta} - f_n^*}{\infty} \leq 2\eps_n} \ \geq \ \exp(-n\eps_n^2), \\ 
	\label{eq:condsieve} \Pi_n \pat{W^{A, \Gamma, \Theta} \notin \B_n} \ \leq \ \exp(-5n\eps_n^2), \\
	\label{eq:condentropy} \log N \pat{ 3\eps_n, \B_n, \nor{\cdot}{\infty}} \ \leq \ n\eps_n^2.
\end{gather}
This is the purpose of the next sections.
The first condition \eqref{eq:priormasscond}, referred to as \textit{prior mass condition}, ensures that the prior puts a sufficient amount of mass around the true parameter. Condition \eqref{eq:condsieve}, called \textit{sieve condition}, forces the sieve $\B_n$ to capture most of the mass of the prior, while the \textit{entropy condition} \eqref{eq:condentropy} constrains its size.
These three conditions map one to one with the conditions of Theorem 2.1 in \cite{Ghosal2000Cvrate}, as shown in \cite{Vaart2008ContrateGaussian} for density estimation and regression with fixed design. 
For regression with random design, we recall in the next section some arguments spread out in the Bayesian literature.

\subsubsection{Regression with random design}\label{sec:RDregress}

Here, we show that Theorem 2.1 in \cite{Ghosal2000Cvrate} can be applied in the regression with random design setting, as soon as Conditions \eqref{eq:priormasscond}, \eqref{eq:condsieve}, and \eqref{eq:condentropy} are satisfied.
The procedure consists in showing that the posterior contracts to the density of a pair $(X_i, Y_i)$ and then retrieving $f^*_n$ from this density.
For a function $f : \U_{d_n} \to \R$, we define $P_f : \U_{d_n} \times \R \to \R_+, \ (x, y) \mapsto G_n(x) \cdot \Phi_{f(x), \sigma} (y)$, where $\Phi_{\mu, \sigma}$ is the density of a univariate Gaussian variable with mean $\mu$ and standard deviation $\sigma$ and $G_n$ is the density of one covariate. Then, the density of one observation $(X, Y)$ under regression with random design is $P_{f_n^*}$.
We first prove that Condition \eqref{eq:priormasscond} implies Condition (2.4) in \cite{Ghosal2000Cvrate} with $C = 1$. Detailed calculations can be found in \cite{Finocchio2023DGP}, Section A.2.
We have to compare the uniform neighborhood of $f_n^*$ with the Kullback-Leibler neighborhood
\[ B_2(P_{f_n^*} ; \eps) := \{ g : \mathrm{KL}(P_{f_n^*}, P_g) \leq \eps^2 , \ V_{2, 0} (P_{f_n^*}, P_g) \leq \eps^2 \}, \]
where $\mathrm{KL} (P_f, P_g) := P_f \cro{ \log (dP_f / dP_g)}$ is the Kullback-Leibler divergence between $P_f$ and $P_g$ and $V_{2, 0}(P_f, P_g) \allowbreak := P_f \cro{ \log (dP_f / dP_g) - \mathrm{KL}(P_f, P_g)}^2$ is the Kullback-Leibler variation.
Using the following identities from \cite{Finocchio2023DGP}:
\begin{align*}
	\mathrm{KL}(P_f, P_g) \ & = \ \frac{1}{2\sigma^2} \nor{f-g}{2, G_n}^2, \\
	V_2(P_f, P_g) \ & := \ P_f \cro{ \log \pat{\frac{dP_f}{dP_g}}^2} \ = \ \frac{1}{\sigma^2} \nor{f-g}{2, G_n}^2 \ + \ \pat{\frac{1}{2\sigma^2} \nor{(f-g)^2}{2, G_n}}^2, \\
	V_{2, 0}(P_f, P_g) \ & = \ V_2 (P_f, P_g) \ - \ \mathrm{KL}(P_f, P_g)^2,
\end{align*}
we deduce that, if $\nor{f-g}{\infty} \leq 2\eps$ with $2\eps < 1$, then 
\begin{align*}
	\mathrm{KL}(P_f, P_g) \ & \leq \ \frac{1}{2\sigma^2} \nor{f-g}{\infty}^2 \ \leq \ \frac{2\eps^2}{\sigma^2}, \\
	V_{2, 0}(P_f, P_g) \ & \leq \ 4 C_\sigma^2 \cdot \eps^2, 
\end{align*}
where $C_\sigma := \sqrt{1/\sigma^2 + 1/(4\sigma^4)}$.
Consequently, according to Condition \eqref{eq:priormasscond}, we have
\[ \Pi_n \pat{B_2 (P_{f_n^*} ; \eps_n)} \ \geq \ \Pi_n \pat{\nor{W^{A, \Gamma, \Theta} - f_n^*}{\infty} \leq 2\frac{\eps_n}{2C_\sigma}} \ \geq \ \exp \pat{-\frac{1}{4C_\sigma^2} n \eps_n^2}. \]
One can remark that for Condition (2.4) in \cite{Ghosal2000Cvrate} to be satisfied, we must have $(4C_\sigma^2)^{-1} \leq 1$ which is the case as soon as $\sigma \leq 2$.

Condition (2.3) in \cite{Ghosal2000Cvrate} is immediately deduced from \eqref{eq:condsieve}.
For Condition (2.4) in the same reference, we use the inequality
\begin{equation}\label{ineq:1}
	h(P_f, P_g) \ \leq \ \frac{1}{2\sigma} \nor{f - g}{\infty}, 
\end{equation}
see again \cite{Finocchio2023DGP} for details.
Then, assuming that $\sigma \geq 1$ and according to \eqref{eq:condentropy}, we have
\[ D \pat{\eps_n, \B_n, h} \ \leq \ N \pat{ \frac{\eps_n}{2}, \B_n, h} \ \leq \ N \pat{ \frac{\eps_n}{2\sigma}, \B_n, h} \ \leq \ N \pat{ \eps_n, \B_n, \nor{\cdot}{\infty}} \ \leq \ \exp (n\eps_n^2), \]
where the first inequality comes from the definition of the packing number $D$ and the covering number $N$ given in Section \ref{sec:Notation} and where the third inequality follows from \eqref{ineq:1}.
Theorem 2.1 in \cite{Ghosal2000Cvrate} then ensures posterior consistency to $P_{f_n^*}$ at rate $\eps_n$ in Hellinger distance.
Now, because we also have the converse inequality
\[ h^2(P_f, P_g) \ \geq \ \frac{1}{4\sigma^2} \exp \pat{-\frac{Q^2}{2\sigma^2}} \cdot \nor{f-g}{2, G_n}^2 \qquad \text{if } \nor{f}{\infty} \leq Q \text{ and } \nor{g}{\infty} \leq Q \]
and because $h(P_{f^Q}, P_{g^Q}) \leq h(P_f, P_g)$ when nothing is assumed on $f$ and $g$ with $f^Q = (f \vee -Q) \wedge Q$, 
we obtain posterior contraction to $f_n^*$ at rate $\eps_n$ in the $L^2(G_n)$-distance:
\[ \Prob_n^* \cro{\Pi_n \pat{g \in \mathcal{C}(\U_{d_n}) : \bignor{{f_n^*}^Q - g^Q}{2, G_n} > D_\sigma^Q \cdot \eps_n \ | \ (X_1, Y_1), \ldots, (X_n, Y_n)}} \tendto{n} 0, \]
where $D_\sigma^Q := M \cdot 2\sigma \cdot \exp \pat{Q^2/(4\sigma^2)}$.

\begin{remark}
	The restriction to $[1, 2]$ for the standard deviation $\sigma$ can be relaxed. In fact, if $\sigma > 2$, then it suffices to consider Theorem 2.1 in \cite{Ghosal2000Cvrate} with $C = (4C^2_\sigma)^{-1}$. Condition (2.4) in \cite{Ghosal2000Cvrate} is then immediately satisfied and, for Condition (2.3), our proof of \eqref{eq:condsieve} can be adapted to replace 5 by $4+C$.
	On the contrary, if $0 < \sigma < 1$, Condition (2.2) in \cite{Ghosal2000Cvrate} can be satisfied by multiplying $\eps_n$ by $\sigma^{-1}$.
\end{remark}

\bigskip

\subsubsection{Prior mass condition \eqref{eq:priormasscond}}

We verify here that $\Pi_n \pat{\nor{W^{A, \Gamma, \Theta} - f_n^*}{\infty} \leq 2\eps_n} \ \geq \ \exp(-n\eps_n^2)$.
We first reduce the problem to deterministic dimension and direction by conditioning with $\Gamma = \dint$ and integrating over $\O_{d_n}$:
\[ \Pi_n \pat{\bignor{W^{A, \Gamma, \Theta} - f_n^*}{\infty} \leq 2\eps_n} \ \geq \ \pi_{\Gamma, n}(\dint) \int _{\O_{d_n}} \Pi_n \pat{\bignor{W^{A, \dint, q} - f_n^*}{\infty} \leq 2\eps_n} dq .\]
Now, we want to bound from below the integrand on a significant subset of $\O_{d_n}$.
We remark that if $q \in \O_{d_n}$ is such that $\nor{f_n^* - \Lambda_{q}(\fint)}{\infty} \leq \eps_n$, then
\begin{equation*}
	\Pi_n \pat{\bignor{W^{A, \dint, q} - f_n^*}{\infty} \leq 2\eps_n} \ \geq \ \Pi_n \pat{\bignor{W^{A, \dint, q} - \Lambda_{q}(\fint)}{\infty} \leq \eps_n}.
\end{equation*}
We will show that the right-hand side is bounded from below by $\exp (-\frac12 n \eps_n^2)$ and then, we will derive a lower bound on the measure of the set of isometries $q \in \O_{d_n}$ satisfying $\nor{f_n^* - \Lambda_q(\fint)}{\infty} \leq \eps_n $.

According to Lemma 5.3 in \cite{Vaart2008RKHS}, for $q \in \O_{d_n}$, we can write
\begin{align*}
	\Pi_n \pat{\bignor{W^{A, \dint, q} - \Lambda_q(\fint)}{\infty} \leq \eps_n} \ & = \ \int _0^\infty \Pi_n \pat{\bignor{W^{a, \dint, q} - \Lambda_q(\fint)}{\infty} \leq \eps_n} \pi_{n, \dint} (a)da \\
	& \geq \ \int _0^\infty \exp \pat{-\phi_{\Lambda_q(\fint)}^{a, \dint, q} (\eps_n/2)} \pi_{n, \dint} (a)da,
\end{align*}
where $\phi_{\Lambda_q(\fint)}^{a, \dint, q}$ is the concentration function introduced in \eqref{rem:concfunct}.
The ingredients for controlling the two terms of the concentration function are Lemmas \ref{lem:cfinf} and \ref{lem:smallballprob}; let us restrict the involved quantities to their range of validity.
Keeping notations of Lemma \ref{lem:cfinf}, let $a \in [T_n, 2T_n]$ where $T_n := \pat{\frac{2C_1^\dint K_n^2}{\eps_n}}^{1/\beta}$.
Up to making $C_1$ larger and assuming that $n$ is large enough, we can suppose that $\eps_n < 8$ and that
\[ T_n \ \geq \ \max \acc{8 ; \ \pat{\frac{4}{C_1^\dint K_n^2}}^{1/\beta} ; \ \frac{\sqrt{\log 2}}{2\sqrt{\dint}} }. \]
We then obtain $C_1^\dint K_n^2 a^{-\beta} \leq \eps_n/2 < 4$ and $a \geq \sqrt{\log 2}/(2\sqrt{\dint})$.

The hypothesis being verified, Lemma \ref{lem:cfinf} with $\eps = \eps_n/2$ writes
\[ \inf \acc{\nor{\overline{h}}{\H_{a, \dint, q}}^2 : \overline{h} \in \H_{a, \dint, q}, \  \nor{\overline{h} - \Lambda_q(\fint)}{\infty} \leq \eps_n/2 } \ \leq \ C_2^\dint K_n^2 \cdot a^{\dint}. \]
Then, using the expression \eqref{rem:concfunct} of the concentration function, a combination of the two lemmas gives
\begin{align*}
	\phi_{\Lambda_q(\fint)}^{a, \dint, q} (\eps_n/2) \ & \leq \ C_2^\dint K_n^2 \cdot a^{\dint} \ + \ C^\dint \dint^{c\dint} \cdot a^{\dint} \log (2a /\eps_n)^{\dint +1} \\
	& \leq \ \pat{ C_2^\dint K_n^2 (\log 2)^{-\dint-1} + C^\dint \dint^{c\dint}} a^{\dint} \log (2a /\eps_n)^{\dint +1},
\end{align*}
where the last inequality holds because $a \geq 8$. 
We can show that there exist constants $\mathcal{C}_1(\beta)$ and $\mathcal{C}_2(\beta)$ depending only on $\beta$ such that 
\begin{gather}
C_2^\dint K_n^2 (\log 2)^{-\dint-1} + C^\dint \dint^{c\dint} \ \leq \ \mathcal{C}_1(\beta) ^\dint K_n^2 \cdot \dint^{c\dint}, \\
\label{eq:4Tn} \log (4T_n / \eps_n) \ \leq \ \mathcal{C}_2(\beta) \cdot \pat{\dint + \log(1/\eps_n) + \log K_n}.
\end{gather}
Using that $1/\eps_n$ and $K_n$ are at most polynomial in $n$ and that $\dint \lesssim \sqrt{\log n}$ (Assumption \ref{ass:growth}), we deduce two more intermediate constants $\mathcal{C}_3(\beta, D_1, D_2)$ depending only on $\beta$, $D_1$, $D_2$ and $\mathcal{C}_4(\beta)$ depending only on $\beta$ for which
\begin{align*}
	\int _{T_n}^{2T_n} \exp \pat{-\phi_{\Lambda_q(\fint)}^{a, \dint, q} (\eps_n/2)} \pi_{n, \dint} (a)&da \ \geq \ \int _{T_n}^{2T_n} \exp \pat{- \mathcal{C}_1(\beta) ^\dint K_n^2 \cdot \dint^{c\dint} \cdot a^\dint \log (2a /\eps_n)^{\dint +1}} \pi_{n, \dint} (a)da \\
	\geq \ & \exp \pat{ - \mathcal{C}_1(\beta) ^\dint K_n^2 \cdot \dint^{c\dint} \cdot (2T_n)^\dint \log (4T_n/\eps_n)^{\dint+1}} \inf _{x \in [T_n, 2T_n]} \pi_{n, \dint} (x) \\
	(\text{Assumption }\ref{ass:resc}) \qquad \geq \ & \exp \pat{- \mathcal{C}_3(\beta, D_1, D_2) ^\dint K_n^2 \cdot \dint^{c\dint} \cdot (2T_n)^\dint \log (4T_n/\eps_n)^{\dint+1}} \\
	\text{by } \eqref{eq:4Tn} \qquad \geq \ & \exp \pat{-\mathcal{C}_4(\beta)^{\dint^2} K_n^{2 + 2\dint/\beta} \cdot \eps_n^{-\dint/\beta} \cdot (\log n)^{\dint +1}}.
\end{align*}
By taking $C_\eps(\beta) \geq (2 \mathcal{C}_4(\beta))^{\beta/(\beta+1)}$ in the expression of $\eps_n$, we can bound the term inside the last inverse exponential by $\frac12 n \eps_n^2$ and achieve
\begin{equation}
	\Pi_n \pat{\bignor{W^{A, \dint, q} - \Lambda_q(\fint)}{\infty} \leq \eps_n} \ \geq \  \exp \pat{-\frac12 n \eps_n^2}.
\end{equation}

At this point, the problem amount to bound from below the measure of the set of isometries $q \in \O_{d_n}$ satisfying $\nor{f_n^* - \Lambda_q(\fint)}{\infty} \leq \eps_n $. We denote by $\mathcal{A}_{\eps_n}$ this set.
The core function $\fint$ was chosen in the Hölder ball $\holder^\beta(\U_\dint ; K_n)$ of radius $K_n$ according to the norm $\holdern{\cdot}$.
Consequently, for all $q, q' \in \O_{d_n}$, 
\begin{align*}
	\nor{\Lambda_{q'}(\fint) - \Lambda_{q}(\fint)}{\infty} \ & = \ \sup _{x \in \U_{d_n}} \abs{\fint \pat{(q' x)_{\bdint}} - \fint \pat{(qx)_{\bdint}}} \ \leq \ \sqrt{\dint} K_n \cdot \vertiii{q' - q}^{1 \wedge \beta},
\end{align*}
where $\vertiii{\cdot}$ is the operator norm induced by the Euclidean norm and where the last inequality is obtained by distinguishing the cases $\beta \leq 1$ and $\beta > 1$.
From now on, it is apparently sufficient to compute the measure of a ball in $\O_{d_n}$ with radius $R_n := (\eps_n/(\sqrt{\dint}K_n))^{1 \vee \beta^{-1}}$. In fact, $B_{\O_{d_n}} \pat{q_n^*, \ R_n} \subset \mathcal{A}_{\eps_n}$. However, this leads to a design dimension $d_n$ not larger than $n^{\dint/(4\beta + 2\dint)}$. To obtain $d_n$ of order $n^{\dint/(2\beta + \dint)}$, we have to consider a larger subset.

\begin{notation}
	Let $F$ be a linear subspace of $\R^{d_n}$. We denote by $\O_{d_n}(F)$ the set of isometries that fix $F$:
	\[ \O_{d_n}(F) \ := \ \{ q' \in \O_{d_n} : q'_{|F} = \Id \}. \]
\end{notation}

\noindent
Then, for all $q' \in \O_{d_n}(\optsub_n)$, we have
\[ \Lambda_{q_n^* q'}(\fint) \ = \ \Lambda_{q_n^*}(\fint) \, \circ \, q' \ = \ f_n^*, \]
and, 
\[ \nor{f_n^* - \Lambda_q(\fint)}{\infty} \ = \ \nor{\Lambda_{q_n^* q'}(\fint) - \Lambda_q(\fint)}{\infty} \ \leq \ \sqrt{\dint}K_n \cdot \vertiii{q^*_n q' - q}^{1 \wedge \beta}. \]

For $\eps > 0$, we define the set
\[ \Q_{q_n^*, \eps} \ := \ \{ q \in \O_{d_n} : \exists q' \in \O_{d_n}(\optsub_n), \ \vertiii{q_n^* q' -q} \leq \eps \}. \]
Then, $\Q_{q^*_n, R_n} \subset \mathcal{A}_{\eps_n}$ and according to Lemma \ref{lem:measure1} in Appendix \ref{sec:Lem}, the mass of $\mathcal{A}_{\eps_n}$ satisfies
\[ \Prob ( \Theta \in \mathcal{A}_{\eps_n}) \ \geq \ \pat{\frac{2}{\pi d_n}}^{\tfrac{\dint}{2}} \cdot \pat{\pat{\frac{\eps_n}{\sqrt{\dint}K_n}}^{1 \vee \beta^{-1}} \frac{1}{16 \sqrt{\dint d_n}}}^{\dint (d_n-1)}.\]
Recall that we have the following lower bound:
\[ \Pi_n \pat{\bignor{W^{A, \Gamma, \Theta} - f_n^*}{\infty} \leq 2\eps_n} \ \geq \ \pi_{\Gamma, n}(\dint) \cdot \Prob ( \Theta \in \mathcal{A}_{\eps_n})  \cdot \exp \pat{-\frac12 n \eps_n^2}.\]
In order to establish the prior mass condition, it suffices to derive the greatest design dimension $d_n$ for which we can reach
\[ \Prob ( \Theta \in \mathcal{A}_{\eps_n}) \ \geq \ \pi_{\Gamma, n} (\dint)^{-1} \exp \pat{-\frac12 n \eps_n^2}. \]
For $n$ and $C_\eps$ large enough, a design dimension $d_n$ as specified in Assumption \ref{ass:growth} is appropriate and leads to Condition \eqref{eq:priormasscond}.


\subsubsection{Sieve condition \eqref{eq:condsieve}}

The second condition can be verified in a similar way to \cite{Jiang2021VariableSelection}.
As in the previous section, we will first treat the case with deterministic rescaling parameter, dimension, and direction and then integrate according to $A$, $\Gamma$, and $\Theta$.

We introduce a sequence $M_n$ with polynomial growth in $n$ and, for $1 \leq b \leq \dintup$, quantities $r_{n, b}$ such that
\begin{equation}\label{eq:rnb}
	r_{n, b}^{b} (\log r_{n, b})^{b+1} = C_r n \eps_n^2,
\end{equation}
for a large constant $C_r$. Assumption \ref{ass:growth} ensures that $r_{n, \dintup}$ tends to infinity with $n$.
The sieve $\B_n$ is defined as follows: 
\[ \B_n \ := \ \bigcup_{q \in \O_{d_n}} \mathcal{B}_{n, q}, \]
with
\[ \mathcal{B}_{n, q} \ := \ \bigcup _{b = 1}^{\dintup} \mathcal{B}_{n, b, q} \qquad \text{and} \quad \mathcal{B}_{n, b, q} \ := \ M_n \sqrt{r_{n, b}} \cdot \H_1^{r_{n, b}, b, q} + \eps_n B_1,\]
where $B_1$ and $\H_1^{a, b, q}$ are the unit balls in the Banach space $(\mathcal{C}^0(\U_{d_n}), \nor{\cdot}{\infty})$ and in the RKHS $\H_{a, b, q}$ respectively.

The nesting property of Lemma 4.7 in \cite{Vaart2009AdaptBayes} remains true in the present setting. That is, for $a \leq a'$, 
\[ \sqrt{a} \cdot \H_1^{a, b, q} \ \subseteq \ \sqrt{a'} \cdot \H_1^{a', b, q}. \]
Consequently, if $1 \leq a \leq r_{n, b}$, then
\[ M_n \H_1^{a, b, q} + \eps_n B_1 \ \subseteq \ M_n \sqrt{\frac{r_{n, b}}{a}} \cdot \H_1^{r_{n, b}, b, q} + \eps_n B_1 \ \subseteq \ \mathcal{B}_{n, b, q}.\]
By Borell's inequality (see \cite{Vaart2008RKHS}, Theorem 5.1, or \cite{Borell1975BrunnMinkowski}), for every $a \in [1, r_{n, b}]$, 
\begin{align*}
	\Pi_n (W^{a, b, q} \notin \B_n) \ & \leq \ \Pi_n (W^{a, b, q} \notin \mathcal{B}_{n, b, q}) \\
	& \leq \ \Pi_n(W^{a, b, q} \notin M_n \H_1^{a, b, q} + \eps_n B_1) \\
	& \leq \ 1 - \Phi \pat{ \Phi ^{-1} \pat{ \Pi_n \pat{\nor{W^{a, b, q}}{\infty} \leq \eps_n}} + M_n},
\end{align*}
where $\Phi$ is the cumulative distribution function of the standard normal distribution.
Now, because
\[ \Pi_n \pat{\nor{W^{a, b, q}}{\infty} \leq \eps_n} \ \geq \ \Pi_n \pat{\nor{W^{r_{n, b}, b, q}}{\infty} \leq \eps_n} \ = \ \exp \pat{-\phi_0^{r_{n, b}, b, q} (\eps_n)}, \]
we have
\[ \Pi_n (W^{a, b, q} \notin \B_n) \ \leq \ 1 - \Phi \pat{ \Phi ^{-1} \pat{ \exp \bigpat{-\phi_0^{r_{n, b}, b, q} (\eps_n)}} + M_n}. \]
For $n$ large enough, we can assume that $\eps_n \leq 4$ and $r_{n, b} \geq \sqrt{\log 2}/(2\sqrt{b})$, so according to Lemma \ref{lem:smallballprob}, there exists a constant $C_1(\beta)$ such that
\begin{equation}
	\phi_0^{r_{n, b}, b, q} (\eps_n) \ \leq \ C^b b^{cb} r_{n, b}^{b} \log \pat{\frac{r_{n, b}}{\eps_n}}^{b+1} \ \leq \ C_1(\beta)^{b+1} b^{cb} (\log n)^{b+1} \cdot n\eps_n^2.
\end{equation}
By taking $M_n^2$ a very large multiple of $C_1^{\dintup+1} \dintup^{c\dintup} (\log n)^{\dintup+1} \cdot n\eps_n^2$, we can reach $M_n \geq 4 \sqrt{\phi_0^{r_{n, b}, b, q}(\eps_n)}$. The second assertion of Lemma 4.10 in \cite{Vaart2009AdaptBayes} gives $M_n \geq -2 \Phi^{-1} \pat{\exp \bigpat{-\phi_0^{r_{n, b}, b, q} (\eps_n)}}$ which leads to the upper bound
\[ \Pi_n (W^{a, b, q} \notin \B_n) \ \leq \ 1 - \Phi( M_n/2) \ \leq \ \exp (-M_n^2/8). \]
Taking into account the random rescaling parameter $A$, we have, for sufficiently large $n$, 
\begin{align*}
	\Pi_n (W^{A, b, q} \notin \B_n) \ & \leq \ \int_{0}^{r_{n, b}} \Pi_n (W^{a, b, q} \notin \B_n) \pi_{n, b}(a)da \ + \ \pi_{n, b}(A \geq r_{n, b}) \\
	(\text{Assumption } \ref{ass:resc}) \quad & \leq \ \exp (-M_n^2/8) \ + \ D_3 \int _{r_{n, b}}^\infty \exp \pat{-D_4 a^{b} (\log a)^{b+1} } da \\
	& \leq \ \exp (-M_n^2/8) \ + \ D_3 \int _{r_{n, b}}^\infty D_4 a^{b-1} ((b+1)\log^b a + b\log^{b+1} a) \exp \pat{-D_4 a^{b} (\log a)^{b+1} } da \\
	& \leq \ \exp (-M_n^2/8) \ + \ D_3 \exp \pat{-D_4 r_{n, b}^{b} (\log r_{n, b})^{b+1}} \\
	& \leq \ \frac12 \exp (-5n\eps_n^2) \ + \ \frac12 \exp (-5n\eps_n^2) \\
	& = \ \exp (-5n\eps_n^2), 
\end{align*}
where the last inequality holds because $C_r$ is supposed to be large enough.

Now considering the prior on the sparsity pattern, we obtain
\begin{align*}
	\Pi_n (W^{A, \Gamma, \Theta} \notin \B_n) \ & = \ \sum _{b = 1}^{\dintup} \Pi_n(\Gamma = b) \int _{\O_{d_n}} \Pi_n(W^{A, b, q} \notin \B_n) dq \ \leq \ \exp (-5n\eps_n^2).
\end{align*}
That completes the proof of Condition \eqref{eq:condsieve}.

\subsubsection{Entropy condition \eqref{eq:condentropy}}

We use the notation and quantities of the previous section and suppose $n$ large enough so that $\eps_n < M_n \sqrt{r_{n, b}}$ and $r_{n, b} \geq 1/(96\sqrt{b})$ for all $b \in \eint {1, \dintup}$. According to Lemma \ref{lem:entropybound}, for all $q \in \O_{d_n}$, the metric entropy of $\mathcal{B}_{n, b, q}$ is bounded as follows:
\begin{align*}
	\log N \pat{2\eps_n, \mathcal{B}_{n, b, q}, \nor{\cdot}{\infty} } \ & = \ \log N \pat{2\eps_n, M_n \sqrt{r_{n, b}} \H_1^{r_{n, b}, b, q} + \eps_n B_1, \nor{\cdot}{\infty} } \\
	& \leq \ \log N \pat{\eps_n, M_n \sqrt{r_{n, b}} \H_1^{r_{n, b}, b, q}, \nor{\cdot}{\infty} }, \\
	& \leq \ C^b b^{4b} r_{n, b}^{b} \log \pat{M_n \sqrt{r_{n, b}} \eps_n^{-1}}^{b+1} .
\end{align*}
Then, the simple estimation $\log \pat{M_n \sqrt{r_{n, b}} \eps_n^{-1}} \lesssim \log n$ together with \eqref{eq:rnb} yields
\begin{equation}\label{eq:entrop_Bnbq}
	\log N \pat{2\eps_n, \mathcal{B}_{n, b, q}, \nor{\cdot}{\infty} } \ \lesssim \ C_r C^b b^{4b} n \eps_n^2 (\log n)^{b+1}.
\end{equation}
To extend these inequalities to the full sieve $\B_n$, we search a finite net $\mathcal{R}_n$ over $\O_{d_n}$ such that,
\begin{equation}\label{eq:Sievenet}
	\B_n \ = \ \bigcup_{q \in \O_{d_n}} \mathcal{B}_{n, q} \ \subseteq \ \bigcup_{\overline{q} \in \mathcal{R}_n} \pat{\mathcal{B}_{n, \overline{q}} + \eps_n B_1}.
\end{equation}
This enables us to bound the $3\eps_n$-entropy of $\B_n$ by the cardinal of the net $\mathcal{R}_n$ times the maximal $2\eps_n$-entropy of sets $\mathcal{B}_{n, b, q}$:
\begin{align*}
	N \pat{3\eps_n, \B_n, \nor{\cdot}{\infty}} \ & \leq \ \sum _{\overline{q} \in \mathcal{R}_n} N \pat{3\eps_n, \mathcal{B}_{n, \overline{q}} + \eps_n B_1, \nor{\cdot}{\infty}}  \\
	& \leq \ \sum _{\overline{q} \in \mathcal{R}_n} N \pat{ 2\eps_n, \mathcal{B}_{n, \overline{q}}, \nor{\cdot}{\infty}} \\
	& \leq \ \sum _{\overline{q} \in \mathcal{R}_n} \sum _{b = 1}^{\dintup} N \pat{2\eps_n, \mathcal{B}_{n, b, \overline{q}}, \nor{\cdot}{\infty}} \\
	& \leq \ \abs{\mathcal{R}_n} \cdot \dintup \max _{\substack{1 \leq b \leq \dintup \\[0.07cm] \overline{q} \in \mathcal{R}_n}} N \pat{2\eps_n, \mathcal{B}_{n, b, \overline{q}}, \nor{\cdot}{\infty} } \\
	\text{by } \eqref{eq:entrop_Bnbq} \quad & \lesssim \ \abs{\mathcal{R}_n} \cdot C^\dintup \dintup^{4\dintup+1} \cdot n \eps_n^2 (\log n)^{\dintup+1}.
\end{align*}
To achieve this, we need the following lemma from \cite{Tokdar2011DimensionAdapt}.
\begin{lemma}[Tokdar 2011, Lemma 1]\label{lem:Tok}
	Let $a > 0$, $b < d_n$, and $q, \tilde{q} \in \O_{d_n}$. Then 
	\[ \H_1^{a, b, q} \ \subseteq \ \H_1^{a, b, \tilde{q}} \ + \ a \sqrt{2b} \cdot \vertiii{q - \tilde{q}} B_1, \]
	where $B_1$ is the unit ball in $(\mathcal{C}^0(\U_{d_n}), \nor{\cdot}{\infty})$.
\end{lemma}

By examining the representation result in \eqref{eq:repr_thm} for $\H_{a, b, q}$, we see that, for all $q' \in \O_{d_n}(q^{-1}(E_\b))$, we have $\H_{a, b, q} = \H_{a, b, qq'}$. Hence, Lemma \ref{lem:Tok} gives
\[ \H_1^{a, b, q} \ \subseteq \ \H_1^{a, b, \tilde{q}q'} \ + \ a \sqrt{2b} \cdot \vertiii{q - \tilde{q}} B_1. \]
If $\mathcal{R}_n$ is a net over $\O_{d_n}$ such that for all $q \in \O_{d_n}$, there exist $q' \in \O_{d_n}\Bigpat{q^{-1}\bigpat{E_{\bdintup}}}$ and $\overline{q} \in \mathcal{R}_n$ with $\vertiii{qq' - \overline{q}} \leq \zeta_n$, where $\zeta_n$ is the minimum of $\eps_n/\Bigpat{M_n r_{n, b}^{3/2} \sqrt{2\dintup}}$ when $b$ runs through $\eint{1, \dintup}$, then
\begin{align*}
	M_n \sqrt{r_{n, b}} \cdot \H_1^{r_{n, b}, b, q} \ & \subseteq \ M_n\sqrt{r_{n, b}} \cdot \H_1^{r_{n, b}, b, \overline{q}} \ + \ M_n r_{n, b}^{3/2} \sqrt{2b} \cdot  \vertiii{qq' - \overline{q}} B_1 \\
	& \subseteq \ M_n \sqrt{r_{n, b}} \cdot \H_1^{r_{n, b}, b, \overline{q}} \ + \ \eps_n B_1 \\
	& = \ \mathcal{B}_{n, b, \overline{q}}.
\end{align*}
This clearly implies $\mathcal{B}_{n, q} \ \subseteq \ \mathcal{B}_{n, \overline{q}} + \eps_n B_1$ and hence \eqref{eq:Sievenet}.
It only remains to bound the cardinal of $\mathcal{R}_n$.
That is the result of Lemma \ref{lem:netEntropy}, which yields 
\[ N \pat{3\eps_n, \B_n, \nor{\cdot}{\infty}} \ \lesssim \ \pat{\frac{\pi \sqrt{\dintup d_n}}{2}}^{\dintup} \pat{\frac{16M_n r_n^{3/2} \sqrt{2} d_n \dintup}{\eps_n}}^{\dintup (d_n + \dintup -2)} C^\dintup \dintup^{4\dintup+1} \cdot n \eps_n^2 (\log n)^{\dintup +1}, \]
where $r_n := \max \{r_{n, b} : b \in \eint{1, \dintup} \} \leq C_r n \eps_n^2$. 
Taking the logarithm of both sides and choosing $n$ and $C_\eps$ sufficiently large leads to the desired result.

\subsection{Proof of Theorem \ref{thm:subspace_recovery}}\label{sec:Proof_thmrecov}

We provide the full proof for the density estimation problem and then show how it can be adapted for regression with random design.

\subsubsection{Case $\Gamma < \din$}
The idea of the proof is to show that the non-constancy of $p_0$ in all directions results in a significant difference (in the Hellinger sense) between the true density $p^*$ and any density that is sparser than $p^*$. If this difference can be bounded from below, then the set of densities that are too sparse is expected to have an almost-null posterior mass as soon as the contraction rate falls below the lower bound.

\def\pti {\tilde{p}}
Let $q \in \O_d$ and let $\pti$ be a density that satisfies the model with parameters $\Gamma$ and $q$. Then, $\pti$ is constant on $q^{-1} (E_{1-\boldsymbol{\Gamma}}) + x$, for any $x \in \U_d$. 
Moreover, the intersection between $\mathcal{S}$ and $q^{-1} (E_{1 -\boldsymbol{\Gamma}})$ is non-null so $\pti_{| \mathcal{S}}$ is constant in at least one direction, say $\boldsymbol{\Delta} \in \mathcal{S}$. We will use Assumption \ref{ass:directdetect} and integrate the Hellinger distance over a small hypercube inside the region where $p^*$ is non-constant in $\boldsymbol{\Delta}$.

Let us introduce the operator
\begin{align*}
	\Psi : \ \R^\din & \ \to \ \mathcal{S} \\
	x & \ \mapsto \ (q^*)^{-1} (x^{\dO}) ,
\end{align*}
where $\dO = \sum_{i = 1}^\din e_i$.
In particular, we have $p^* \circ \Psi = p_0$.
We use the notation of Assumption \ref{ass:directdetect} with $\Psi^{-1}(\boldsymbol{\Delta})$ instead of $\v$.

Let $(\boldsymbol{\Delta}, u_1, \ldots, u_{\din -1} ; v_1, \ldots, v_{d-\din})$ be an orthonormal basis adapted to the direct sum $\R^d = \allowbreak  \Delta \allowbreak \oplus \allowbreak  (\Delta^\perp \cap \mathcal{S}) \allowbreak \oplus \allowbreak \mathcal{S}^\perp$, where $\Delta := \vect (\boldsymbol{\Delta})$, and let $R$ be a solid hypercube with edges parallel to this basis, of size $L/\sqrt{d}$ and centered on $\Psi(o)$.
Then, $R \subset B_{L/2}(\Psi(o))$ and the inequality of Assumption \ref{ass:directdetect} is valid when $t \in \Psi^{-1}(R)$.
Considering the basis previously introduced, integrating over $R$ amounts to integrate with respect to each variables. To simplify, we bundle these variables in three groups: a variable $\delta$ parallel to $\Delta$, a variable $u$ parallel to $\Delta^\perp \cap \mathcal{S}$ and a variable $v$ parallel to $\mathcal{S}^\perp$.
In this coordinate system, we can write $\Psi(o) = (\Psi(o)_1, \Psi(o)_2, 0)$ and $p^*(\delta, u, v) = p_0(\Psi^{-1}(\delta, u, 0))$. 
Then 
\begin{equation}\label{disp:hellingerror}
\begin{aligned}
	h^2(p^*_{|R} \, ; \, \pti _{|R}) \ = \ & \iiint _R \abs{\sqrt{p^*(\delta, u, v)} - \sqrt{\pti(0, u, v)}}^2 d\delta \, du \, dv \\
	\ = \ & \iint \pat{\int \abs{\sqrt{ p_0(\Psi^{-1}(\delta, u, 0))} - \sqrt{\pti (0, u, v)}}^2 d\delta } du \, dv \\
	\ = \ & \iint h^2 \pat{{p_0}_{|I_u} \, ; \, \pti (0, u, v) } du \, dv, 
\end{aligned}
\end{equation}
where $I_u$ is the inverse image via $\Psi$ of the range of the integral in $\delta$. Hence 
\begin{align*}
	\Psi(I_u) \ = \ (\Psi(o)_1, u, 0) \ + \ \Big] - \frac{L}{2\sqrt{d}} \boldsymbol{\Delta} \, ; \, \frac{L}{2\sqrt{d}} \boldsymbol{\Delta} \Big[ \qquad \text{with } u \in \Psi(o)_2 \ + \ \Big] - \frac{L}{2\sqrt{d}} \, ; \, \frac{L}{2\sqrt{d}} \Big[ ^{\, \din -1}.
\end{align*}
Then because $\Psi^{-1}(\Psi(o)_1, u, 0) \in B_{L/2}(o)$, there exists $t \in B_{L/2}(o)$ such that
\begin{equation}
	I_u \ =  \ t \ + \ \Big ] - \frac{L}{2\sqrt{d}} \Psi^{-1} (\boldsymbol{\Delta}) \, ; \, \frac{L}{2\sqrt{d}} \Psi^{-1} (\boldsymbol{\Delta}) \Big[.
\end{equation}
Now we can use Assumption \ref{ass:directdetect} and bound from below the Hellinger distance in the last integral, which gives  
\begin{align*}
	h^2(p^*_{|R} \, ; \, \pti _{|R}) \ \geq \ \iint D \cdot \pat{\frac{L}{\sqrt{d}}}^3 \, du \, dv  \ = \ D \cdot \pat{\frac{L}{\sqrt{d}}}^{d+2} .
\end{align*}
Finally, 
\[ \Pi_n (\Gamma < \din \ | \ X_1, \ldots, X_n) \ \leq \ \Pi_n \pat{f \in \mathcal{C}(\U_{d}) : h(p_{f}, p^*) > \eps _n \ | \ X_1, \ldots, X_n}, \]
as soon as the contraction rate achieves $\eps_n \leq \sqrt{D} \pat{\frac{L}{\sqrt{d}}}^{\frac{d+2}{2}}$.


\bigskip

\subsubsection{Case $\Gamma = \din$}

\paragraph{Case $\boldsymbol{\Gamma = \din}$, with $\mathbf{d = 2}$ and $\mathbf{\din = 1}$.}

\def\nonoptsub {{\mathcal{S}'}}

To simplify the presentation, we first restrict ourselves to the case $d = 2$ and $\din = 1$ and consider only the first loss $d_1$. The result with losses $d_2$ and $d_3$ will be obtained using Lemma \ref{lem:dist}.
Assumption \ref{ass:directdetect} specializes as follows: for all $0 < l \leq L$, there exists $o \in [-1+L, 1-L]$ such that, for all $t \in [o-l/2, o+l/2]$ and all constants $c > 0$, 
\[ h^2 \pat{{p_0}_{|]t - \frac{l}{2} ; t + \frac{l}{2}[} \, ; \, c} \ = \ \int_{ - \frac{l}{2}}^{\frac{l}{2}} \abs{ \sqrt{p_0(t+\lambda)} - \sqrt{c}}^2 d\lambda \ \geq \ D \cdot l^3 . \]

We use the fact that the non-constancy of $p^*$ over $\mathcal{S}$ induces a non-constancy over any one-dimensional space not parallel to $\mathcal{S}^\perp$. It is then possible to set a lower bound on the Hellinger distance between $p^*$ and any density that is constant on a space not parallel to $\mathcal{S}^\perp$.
For $q \in \O_2$, we denote $\nonoptsub := q^{-1}(E_{\dO})$. 
If $q$ is not in $\Q_\optsub$, then there exists $0 < \vartheta \leq \pi/2$ such that for all $\overline{q} \in \Q_\optsub$, we have $\vertiii{\overline{q} - q} > \vartheta$.
Then, the intersections of $\nonoptsub^\perp$ and $\mathcal{S}^\perp$ with the unit circle are separated by at least $\vartheta$.

\def\u {\mathbf{u}}
\def\v {\mathbf{v}}
\def\pr {\mathrm{pr}}
With this setting, any square of size $L/\sqrt{2}$ centered in $\Psi(o)$ is included in $\U_2$.
Let $R$ be a solid square of size $L/\sqrt{2}$, parallel to the line $\nonoptsub^\perp$ and centered on $\Psi(o)$.
The line $\nonoptsub^\perp + \Psi(o)$ intersects the border of $R$ at two points (see Figure \ref{fig:2dim}). Using arguments from geometry on the two-dimensional Euclidean space, we can show that the orthogonal projections of these points over $\mathcal{S}$ are at a distance $\zeta \geq \frac{L \vartheta}{4\sqrt{2}} \sqrt{4 - \vartheta^2}$ from $\Psi(o)$. 
Similarly, the line $\nonoptsub + \Psi(o)$ intersects the border of $R$ at two points whose orthogonal projections on $\mathcal{S}$ are at a distance $\chi \leq \frac{L}{2\sqrt{2}} \sqrt{1 - \vartheta^2 + \vartheta^4 /4}$ from $\Psi(o)$.

\begin{figure}[h!]
	\centering
	\includegraphics[scale=0.8]{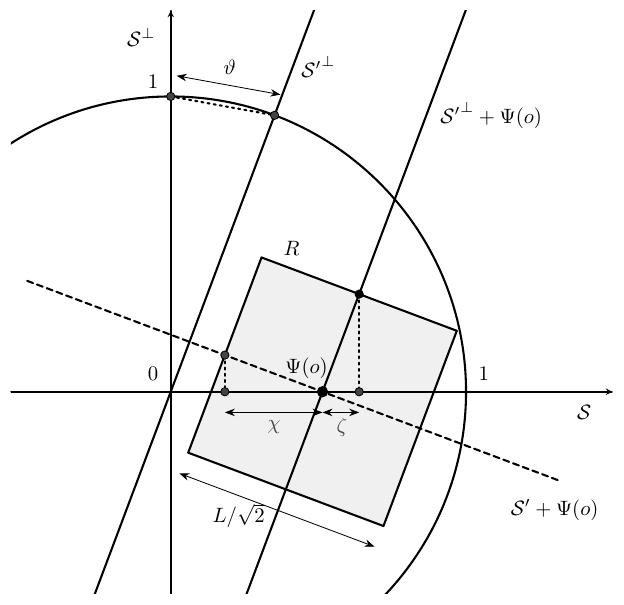}
	\caption{Illustration of the proof of Theorem \ref{thm:subspace_recovery} in the case $\Gamma = \din$ with $d = 2$ and $\din = 1$.}
	\label{fig:2dim}
\end{figure}

Let $(\u, \v)$ be an orthonormal basis of $\R^2$ adapted to the decomposition $\nonoptsub \oplus \nonoptsub^\perp$ and such that $\pr _{\mathcal{S}}(\u) = \frac{2 \sqrt{2}}{L} \chi \cdot \Psi(1)$ and $\pr _{\mathcal{S}}(\v) = \frac{2 \sqrt{2}}{L} \zeta \cdot \Psi(1)$. 
In this system of coordinates, $\Psi(o)$ can be written $(o_1, o_2)$ and for all $u, v \in \R^2$, we have
\[ \Psi^{-1} \pat{\pr _{\mathcal{S}} (u, v)} \ = \ \chi \cdot \frac{2 \sqrt{2}}{L} u \ + \ \zeta \cdot \frac{2 \sqrt{2}}{L} v. \]
We will also use the fact that $p^*(u, v) =  p_0 \pat{\Psi^{-1} \pat{\pr _{\mathcal{S}} (u, v)}}$. 
Then, for all density $\pti$ constant in the direction $\nonoptsub^\perp$, we have
\begin{align*}
	h^2 (p^*_{|R} \, ; \, \pti_{|R}) \ = \ & \iint _R |\sqrt{p^*(u, v)} - \sqrt{\pti(u, 0)}|^2 du \, dv \\
	\ = \ & \iint _R |\sqrt{p_0 \pat{\Psi^{-1} \pat{\pr _{\mathcal{S}} (u, v)}}} - \sqrt{\pti(u, 0)}|^2 du \, dv \\
	\ = \ & \int _{o_1 - L/(2\sqrt{2})}^{o_1 + L/(2\sqrt{2})} \int _{o_2 - L/(2\sqrt{2})}^{o_2 + L/(2\sqrt{2})} \abs{ p_0 \Bigpat{\chi \cdot \frac{2 \sqrt{2}}{L} u \ + \ \zeta \cdot \frac{2 \sqrt{2}}{L} v}^{1/2} - \sqrt{\pti(u, 0)}}^2 du \, dv \\
	\ = \ & \int _{-L/(2\sqrt{2})}^{L/(2\sqrt{2})} \int _{-L/(2\sqrt{2})}^{L/(2\sqrt{2})} \abs{ p_0 \Bigpat{o + \chi \cdot \frac{2 \sqrt{2}}{L} u \ + \ \zeta \cdot \frac{2 \sqrt{2}}{L} v}^{1/2} - \sqrt{\pti(u, 0)}}^2 dv \, du \\
	\ = \ & \int _{-L/(2\sqrt{2})}^{L/(2\sqrt{2})} \frac{L}{2\zeta \cdot \sqrt{2}} \pat{\int _{-\zeta}^{\zeta} \abs{ p_0 \Bigpat{o + \chi \cdot \frac{2 \sqrt{2}}{L} u \ + \ w}^{1/2} - \sqrt{\pti(u, 0)}}^2 dw }\, du \\
	(\text{Assumption }\ref{ass:directdetect}) \quad \geq \ &  \int _{-L/(2\sqrt{2})}^{L/(2\sqrt{2})} \frac{L}{2\zeta \cdot \sqrt{2}} \cdot D \cdot 8\zeta^3 \, du \ = \ 2D L^2 \cdot \zeta^2 \ \geq \ D \cdot \frac{L^4}{16} \vartheta^2 (4 - \vartheta^2).
\end{align*}
Finally, 
\[ \Pi_n \pat{\Gamma = 1 \text{ and } \min_{q \in \Q_\optsub} \vertiii{\Theta - q} \geq \vartheta \ | \ X_1, \ldots, X_n} \ \leq \ \Pi_n \pat{f \in \mathcal{C}(\U_{d}) : h(p_{f}, p^*) > \eps _n \ | \ X_1, \ldots, X_n}, \]
as soon as $\eps_n < \sqrt{ 2D L^2 \cdot \zeta^2}$.

\bigskip

\paragraph{Case $\boldsymbol{\Gamma = \din}$, with arbitrary $\mathbf{d > \din}$.}
Given a non-optimal subspace $\optsub'$, we need to quantify how far from $\mathcal{S}^\perp$ the orthogonal complement $\optsub'^\perp$ is. This result, elementary when $d = 2$, is stated for arbitrary $d > \din$ in the following lemma. A proof is given in Appendix \ref{sec:Lem}.
\def\thetbar {\overline{\vartheta}}
\def\nonoptsub {{\mathcal{S}'}}
\begin{lemma}\label{lem:dist}
	Let $q \in \O_d$ and denote $\nonoptsub := q^{-1}(E_{\dO})$ and $\mathbb{S}_d := \{x \in \R^d : \nor{x}{} = 1\}$. 
	For $\vartheta > 0$, if we suppose respectively, 
	\begin{align}
		\label{eq:lemIsonorm} \forall \ \overline{q} \in \Q_\optsub, \quad \vertiii{\overline{q} - q} \ > \ \vartheta, \\
		\label{eq:lemHausdist} \dhaus(\mathcal{S} \cap \mathbb{S}_d, \nonoptsub \cap \mathbb{S}_d) \ > \ \vartheta, \\
		\label{eq:lemFrobnorm} \nor{P_\mathcal{S} - P_\nonoptsub}{F} \ > \ \vartheta,
	\end{align}
	then there exists $r \in \nonoptsub^\perp \cap \mathbb{S}_d$, such that the distance between $r$ and $\mathcal{S}^\perp \cap \mathbb{S}_d$ is at least $\thetbar$, with $\thetbar = \vartheta$ in the first two cases \eqref{eq:lemIsonorm}, \eqref{eq:lemHausdist} and $\thetbar = \vartheta/\sqrt{2 \din}$ for the case \eqref{eq:lemFrobnorm}.
\end{lemma}

Now we work under the assumptions of Lemma \ref{lem:dist}.
Let $F$ be the linear span of $q^{-1}(r)$ and its orthogonal projection $\boldsymbol{\Lambda}$ on $\mathcal{S}^\perp$ (or any vector of $\mathcal{S}^\perp$ if the orthogonal projection is zero). Then $F$ has a non-zero intersection with $\mathcal{S}$. Let $\Delta$ be this one-dimensional intersection.

Let $R$ be a solid hypercube centered on $\Psi(o)$, with size $\overline{L} := L/\sqrt{d}$, and aligned with $q^{-1}(r), u_1, \ldots, u_{\din - 1}$ and $v_1, \ldots, v_{d - \din - 1}$ where $(\boldsymbol{\Delta}, u_1, \ldots, u_{\din-1}, \boldsymbol{\Lambda}, v_1, \ldots, v_{d - \din - 1})$ is an orthogonal basis adapted to the direct sum $\R^d = \mathcal{S} \oplus \mathcal{S}^\perp$. With the restrictions on $o$, $R$ is included in $\U_d$.

We will bound from below the quantity $h^2 (p^*_{|R} ; \pti_{|R})$ by using the preceding two-dimensional case on slices of $R$.
For $t \in \{ 0 \} \times \prod_{i = 1}^{\din-1} [o - \overline{L}/2 \cdot u_i ; o +\overline{L}/2 \cdot u_i] \times \{ 0 \} \times \prod _{j = 1}^{d - \din - 1} [o - \overline{L}/2 \cdot v_j ; o +\overline{L}/2 \cdot v_j]$, the plane $F + t$ contains one element parallel to $\mathcal{S}$ and one element parallel to $\mathcal{S}^\perp$, so the situation is analogue to the previous case, replacing $\zeta$ by $\overline{\zeta} := \frac{\overline{L}}{4} \overline{\vartheta} \sqrt{4 - \overline{\vartheta}^2}$ (Figure \ref{fig:dimd}). 
With all this in mind, for all density $\pti$ constant in the direction $q^{-1}(r)$, one has
\begin{equation}\label{eq:Hellingerrd}
	h^2 (p^*_{|R} \, ; \, \pti_{|R}) \ = \ \int _t h^2 (p^*_{|R\cap (G+t)} ; \pti_{|R\cap (G+t)}) dt \ \geq \ \int _t 4D\overline{L}^2  \overline{\zeta}^2 dt \ = \  4D\overline{L}^d \overline{\zeta}^2,
\end{equation}
which is sufficient to conclude. 
Statement \eqref{eq:recov_supd0} is valid as soon as $\eps_n < \sqrt{4D\overline{L}^d \overline{\zeta}^2}$ or alternatively with $\overline{\vartheta} \geq \sqrt{2/D} \pat{\sqrt{d}/L}^{(d+2)/2}\eps_n$, provided that the right side of the inequality is less than one, which is the case when $n$ is sufficiently large.

The case $\Gamma > \din$ can be proven in a similar way.

\begin{figure}[h!]
	\centering
	\includegraphics[scale=0.8]{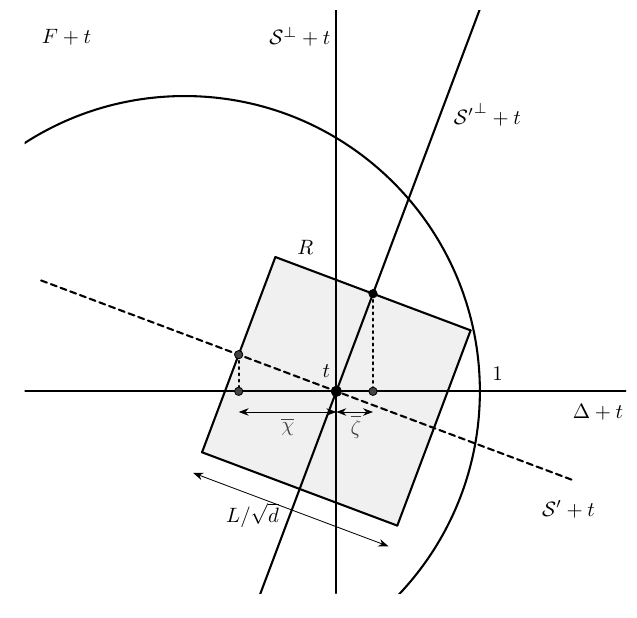}
	\caption{Illustration of the proof of Theorem \ref{thm:subspace_recovery} in the case $\Gamma = \din$ for arbitrary $d > \din$.}
	\label{fig:dimd}
\end{figure}

\subsubsection{Regression with random design.}

\def\fti {\tilde{f}}
\def\Gmin {\underline{G}}
The case of regression with random design is similar. We suppose without loss of generality that $\fin = \fin^Q$ and replace $p^*$ by $f^*$, $p_0$ by $\fin$ and $\pti$ by $\fti$ in the preceding proof. Then, $f^* \circ \Psi = \fin$ and the display \eqref{disp:hellingerror} rewrites
\begin{equation}
	\begin{aligned}
		\nor{f^*_{|R} - \fti _{|R}}{2, G}^2 \ = \ & \iiint _R \abs{f^*(\delta, u, v) - \fti(0, u, v)}^2 G(\delta, u, v) d\delta \, du \, dv \\
		\ \geq \ & \Gmin \iint \pat{\int \abs{\fin(\Psi^{-1}(\delta, u, 0)) - \fti (0, u, v)}^2 d\delta } du \, dv \\
		\ = \ & \Gmin \iint \nor{\fin_{|I_u} - \fti (0, u, v) }{2}^2 du \, dv, \\
		\ \geq \ & \Gmin D \cdot \pat{\frac{L}{\sqrt{d}}}^{d+2},
	\end{aligned}
\end{equation}
where $\Gmin$ is the minimum of the density $G$ over $\U_d$. This leads to
\[\Pi_n (\Gamma < \din \ | \ (X_1, Y_1), \ldots, (X_n, Y_n)) \ \leq \ \Pi_n \pat{f \in \mathcal{C}(\U_{d}) : \bignor{f^Q - {f^*}^Q}{2, G} > \eps _n \ | \ (X_1, Y_1), \ldots, (X_n, Y_n)}, \]
as soon as the contraction rate achieves $\eps_n \leq \sqrt{\Gmin D} \pat{\frac{L}{\sqrt{d}}}^{\frac{d+2}{2}}$.
The same idea leads to Statement \eqref{eq:recov_supd0}, with rates $\delta^{(i)}_n$ divided by $\Gmin$.

\subsection{Lemmas}\label{sec:Lem}

\subsection{Standard lemmas for contraction}

We first give specialized versions of classical lemmas from Bayesian nonparametrics with Gaussian processes, namely Lemmas 4.3, 4.5, and 4.6 in \cite{Vaart2009AdaptBayes}.
Thanks to the meticulous work proposed in \cite{Castillo2024DHGP} with Lemmas 2, 3 and 4, the dimension is made explicit to fit the growing intrinsic dimension framework. 
The next three results are a direct translation of these lemmas and their proofs can be omitted if we keep in mind that for $b \leq d$ two natural numbers, $\U_b \subset [-1, 1]^b$, and that the operator $\Lambda_{b,q}$ is an isometry from $\mathcal{C}(\U_b, \nor{\cdot}{\infty})$ to $\mathcal{C}(\U_d, \nor{\cdot}{\infty})$ and also from $\widetilde{\H}_a$ to $\H_{a, b, q}$.

\begin{lemma}\label{lem:cfinf}
	Let $b,d \in \N^*$, $b \leq d$ and $\beta > 0$. If $g \in \holder^\beta (\U_b ; K)$ for $K > 0$, then, for all $a > 0$ and $q \in \O_d$, there exist constants $C_1(\beta) = C_1$ and $C_2(\beta) = C_2$ that depend only on $\beta$ such that 
	\[ \inf \acc{ \nor{\overline{h}}{\H_{a, b, q}}^2 \ : \ \overline{h} \in \H_{a, b, q}, \  \nor{\overline{h} - \Lambda_{b,q}(g)}{\infty} \leq C_1^b K^2 \cdot a^{-\beta} } \ \leq \ C_2^b K^2 \cdot a^{b}. \]
\end{lemma}

\begin{lemma}\label{lem:entropybound}
	Let $b,d \in \N^*$, $b \leq d$ and $q \in \O_d$. Assume $a \geq 1/(96 \sqrt{b})$. Then, there exists a universal constant $C$ such that, for $\eps < 1$, 
	\[ \log N( \eps, \H_1^{a, b, q}, \nor{\cdot}{\infty}) \ \leq \ C^b b^{4b} \cdot a^b \pat{\log \frac1\eps}^{b+1}. \]
\end{lemma}

\begin{lemma}\label{lem:smallballprob}
	Let $b,d \in \N^*$, $b \leq d$ and $q \in \O_d$. Then, there exist universal constants $c$ and $C$ such that for all $a \geq \sqrt{\log 2}/(2\sqrt{b})$ and $0 < \eps \leq 4$, 
	\[ -\log \Prob \pat{\nor{W^{a, b, q}}{\infty} \leq \eps} \ \leq \ C^b b^{cb} \cdot a^b \pat{\log \frac{a}{\eps}}^{b+1}. \]
\end{lemma}

\subsection{Lemmas for Section \ref{sec:sub_recov} and Appendix \ref{sec:Proof_thmrecov}}

The next lemma gives an intuitive condition for Assumption \ref{ass:directdetect} to be satisfied. It is followed by two proofs of intermediate results from Section \ref{sec:sub_recov}.

\begin{lemma}\label{lem:gradient}
	If $p_0$ and $\fin^Q$ are continuously differentiable over $\U_{\din}$ and if there exist a small $\delta > 0$ and $\din$ points $x_1, \ldots, x_\din$ in $B_{1-\delta}(0)$ such that the gradients at these points are linearly independent, then Assumption \ref{ass:directdetect} holds.
\end{lemma}
\begin{proof}[Proof of Lemma \ref{lem:gradient}]
	Let us begin with the assumption on the regression function. Without loss of generality, we suppose that $\fin = \fin^Q$.
	
	\def\u {\bold u}
	
	Let $\v \in \R^\din$, $\nor{\v}{} = 1$ be a direction and define $k_\v$ the index of the point whose gradient is most aligned with $\v$:
	\[ k_\v \ \in \ \argmax_{i = 1, \ldots, \din} \abs{\dua{\frac{\nabla \fin (x_i)}{\nor{\nabla \fin (x_i)}{}}, \v}} .\]
	Conversely, given $k \in \eint{1, \din}$, define the non-empty set $\mathcal{Y}_k := \{ \u \in \R^\din, \nor{\u}{} = 1, k_\u = k \}$.
	Then, the linear independence of the gradients ensures that the quantities
	\[ b_k \ := \ \inf_{\u \in \mathcal{Y}_k} \abs{\dua{\frac{\nabla \fin (x_k)}{\nor{\nabla \fin (x_k)}{}}, \u}} \ > \ 0,\]
	are strictly greater than zero.
	Consequently, for all $\v \in \R^\din$, $\nor{\v}{} = 1$, 
	\[ \abs{\dua{\nabla \fin (x_{k_\v}), \v}} \ \geq \ b_{k_\v} \nor{\nabla \fin (x_{k_\v})}{}.\]
	Now, by continuity of the differential of $\fin$, there exists for each $k \in \eint{1, \din}$ an open neighborhood $\mathcal{V}_k$ of $x_k$ such that for all $x \in \mathcal{V}_k$, 
	\[ \nor{\nabla \fin (x) - \nabla \fin (x_k)}{} \ \leq \ \frac12 b_k \cdot \nor{\nabla \fin (x_k)}{}. \]
	Moreover, these neighborhoods contain open balls $B_{\delta_k}(x_k) \subset \mathcal{V}_k$ with radius $\delta _k$ depending only on $\fin$, $\mathcal{X}$, and $k$.
	As a result, for $k \in \eint{1, \din}$, for all $\v \in \mathcal{Y}_k$, and for all $x \in \mathcal{V}_k$,
	\begin{align*}
		\abs{\inner{\nabla \fin (x)}{\v}} \ \geq \ & \bigabs{\abs{\inner{\nabla \fin (x) - \nabla \fin (x_k)}{\v}} - \abs{\inner{\nabla \fin (x_k)}{\v}}} \\ \geq \ & \frac12 \min _{i = 1, \dots, \din} \pat{ b_i \cdot \nor{\nabla \fin (x_i)}{}} \ =: \ r(\mathcal{X}).
	\end{align*}
	
	After these preliminary calculations, set $\v \in \R^d$ a direction, $\nor{\v}{} = 1$, and choose $o := x_{k_\v}$ and $L := \delta \wedge \min _{i = 1, \ldots, \din} \delta _i$.
	Then, for all $0 < l \leq L$ and for all $t \in B_{L/2}(o)$, we have, 
	\begin{equation}\label{eq:ineqonJ}
		\abs{\fin(o + t + y \v) - \fin(o+t)} \ \geq \ \abs{y} \cdot r(\mathcal{X}), \quad \forall y \in \ ]-l/2 , l/2[.
	\end{equation}
	Let $c \in \R$ and define the interval 
	\begin{align*}
		J \ := \ & [ o+t , o+t+ s \cdot l/2 \cdot \v[, \\ 
		\text{with} \quad s \ := \ & \sgn \pat{\fin(o+t+ l/2 \cdot \v) - \fin(o+t)} \cdot \sgn \pat{\fin(o+t) - c},
	\end{align*}
	where $\sgn(x) = 1$ if $x \in [0, + \infty[$ and $\sgn(x) = -1$ if $x \in \ ]-\infty, 0[$. The interval $J$ is such that $\abs{\fin (y) - c} \geq \abs{\fin (y) - \fin (o+t)}$, for all $y \in J$.
	Hence, 
	\begin{align*}
		\bignor{\fin _{|I} - c}{2}^2 \ \geq \ & \int _J \abs{\fin (z) - c}^2 dz \\
		\ \geq \ & \int _J \abs{\fin (z) - \fin (o+t)}^2 dz \\
		\text{by \eqref{eq:ineqonJ}} \qquad \geq \ & \int _0^{l/2} |y|^2 r(\mathcal{X})^2 dy \ = \ r(\mathcal{X})^2 \frac{l^3}{24}. 
	\end{align*}
	Consequently, Assumption \ref{ass:directdetect} holds with $D := r(\mathcal{X})^2/24$.
	
	For the density assumption, replace $\fin$ by $p_0$ in the above calculations until the introduction of the constant $c$, which we now choose to be positive.
	Up to replacing $\v$ by $-\v$, we can assume that $p_0(o+t+ l/2 \cdot \v) > p_0(o+t)$.
	Then, if $c > p_0(o+t)$, we have for all $-l/2 < y \leq 0$:
	\begin{align*}
		\abs{\sqrt{p_0(o+t+ y\v)} - \sqrt{c}} \ \geq \ & \abs{\sqrt{p_0(o+t+ y\v)} - \sqrt{p_0(o+t)}} \\
		\text{by \eqref{eq:ineqonJ}} \qquad \geq \ & |y| \cdot r(\mathcal{X}) \frac{1}{2\sqrt{p_0(o+t)}} \ \geq \ \frac{|y| \cdot r(\mathcal{X})}{2 \max_{x \in \U_\din} \sqrt{p_0(x)}},
	\end{align*}
	where the second inequality follows from the mean value theorem and from inequality \eqref{eq:ineqonJ}.
	Similarly, if $c \leq p_0(o+t)$, we have for all $0 \leq y < l/2$:
	\begin{align*}
		\abs{\sqrt{p_0(o+t+ y\v)} - \sqrt{c}} \ \geq \ & \abs{\sqrt{p_0(o+t+ y\v)} - \sqrt{p_0(o+t)}} \\
		\text{by \eqref{eq:ineqonJ}} \qquad \geq \ & |y| \cdot r(\mathcal{X}) \frac{1}{2\sqrt{p_0(o+t+y\v)}} \ \geq \ \frac{|y| \cdot r(\mathcal{X})}{2 \max_{x \in \U_\din} \sqrt{p_0(x)}}.
	\end{align*}
	Finally, 
	\[ h^2({p_0}_{|I} ; c) \ = \ \int_I \abs{\sqrt{p_0(z)} -\sqrt{c}}^2 dz \ \geq \ \frac{r(\mathcal{X})^2}{4 \max_{x \in \U_\din} p_0(x)} \int_0^{l/2} y^2 dy \ = \ \frac{r(\mathcal{X})^2}{96 \max_{x \in \U_\din} p_0(x)} \cdot l^3, \]
	and Assumption \ref{ass:directdetect} holds with $D := r(\mathcal{X})^2/(96 \max_{x \in \U_\din} p_0(x))$.
\end{proof}

\begin{proof}[Proof of Lemma \ref{lem:dist}]
	
	We begin with the Hausdorff distance case \eqref{eq:lemHausdist}. Following the definition \eqref{def:disthaus} of the Hausdorff distance, there exists $x \in \mathcal{S} \cap \mathbb{S}_d$ such that $d(x, \nonoptsub \cap \mathbb{S}_d) > \vartheta$. The transition to the orthogonal complement is straightforward using the notion of \textit{principal angles} between subspaces. There are many characterizations of principal angles but we are more interested in the one given in \cite{Bhatia1997MatrixAnalysis}, Exercise VII.1.12, p 202. The largest principal angle between $\mathcal{S}$ and $\nonoptsub$ is defined as
	\[ \theta_1 \ = \ \max _{\substack{y \in \mathcal{S}\\ \nor{y}{} = 1}} \min _{\substack{z \in \nonoptsub\\ \nor{z}{} = 1}} \arccos \pat{\transp{y} \cdot z}. \]
	Consequently, $\theta_1 > \arccos (1 - \vartheta^2/2)$.
	The principal angles between $\mathcal{S}$ and $\nonoptsub$ and between their orthogonal complements $\mathcal{S}^\perp$ and $\nonoptsub^\perp$ are essentially the same (see \cite{Knyazev2007PrincipalAngles}, Theorem 2.7). In particular, the largest principal angle between $\mathcal{S}^\perp$ and $\nonoptsub^\perp$ is greater than $\arccos (1 - \vartheta^2/2)$, hence there exists $r \in \mathcal{S}^\perp \cap \mathbb{S}_d$ such that $d(r, \nonoptsub^\perp \cap \mathbb{S}_d) > \vartheta$.
	The hypothesis \eqref{eq:lemIsonorm} implies the preceding one \eqref{eq:lemHausdist} and leads to the same conclusion.
	
	Now, suppose that the Frobenius norm between the orthogonal projections over $\optsub$ and $\nonoptsub$ is at least $\vartheta$ (hypothesis \eqref{eq:lemFrobnorm}). We use the fact that the square of the previous Frobenius norm is twice the sum of the squared sines of the principal angles between the two subspaces. To see this, let $M_\optsub$ and $M_\nonoptsub$ be matrices whose columns form an orthonormal basis of $\optsub$ and $\nonoptsub$ respectively. Then $M_\mathcal{S} \transp{M_\mathcal{S}} = P_\mathcal{S}$ is the orthogonal projection over $\mathcal{S}$ and
	\[ \nor{P_{\mathcal{S}} - P_{\mathcal{S'}}}{F}^2 \ = \ 2 \pat{\din - \nor{\transp{M_{\mathcal{S}}} M_{\mathcal{S'}}}{F}^2} \ = \ 2 \pat{\din - \sum_{i = 1}^{\din} \sigma_i^2 \pat{\transp{M_{\mathcal{S}}} M_{\mathcal{S'}}}},\]
	where the $\sigma_i$'s are the singular values of $\transp{M_{\mathcal{S}}} M_{\mathcal{S'}}$, which also coincide with the cosines of the principal angles between $\optsub$ and $\nonoptsub$ (see for example Remark 7.2 in \cite{Vidal2016PCA}). This gives
	\[ \nor{P_{\mathcal{S}} - P_{\mathcal{S'}}}{F}^2 \ = \ 2 \pat{d^* - \sum_{i = 1}^{\din} \cos^2(\theta_i)} \ = \ 2\sum_{i = 1}^{\din} \sin^2(\theta_i), \]
	where the $\theta_i$'s are the principal angles between $\mathcal{S}$ and $\mathcal{S}'$ in non-increasing order. Hypothesis \eqref{eq:lemFrobnorm} then leads to $\sin \theta_1 \geq \vartheta / \sqrt{2\din}$. In this case, the transition to orthogonal complement is straightforward using that
	\[ \nor{P_\optsub - P_\nonoptsub}{F} \ = \ \nor{I- P_\nonoptsub - (I - P_\optsub)}{F} \ = \ \nor{P_{\nonoptsub^\perp} - P_{\optsub^\perp}}{F}.\]
	So the sinus of the largest principal angle between $\optsub^\perp$ and $\nonoptsub^\perp$ is greater than $\vartheta / \sqrt{2\din}$, hence the conclusion.
	\qedhere
\end{proof}

\begin{proof}[Proof of \eqref{eq:subspaceminmax}]

	To see this equality, choose $\dim \optsub' - \din$ vectors $u_1, \ldots, u_{\dim \optsub' - \din}$ in $\optsub^\perp \cap \optsub'$ and consider $\overline{\optsub}_{\min} := \vect (\optsub, u_1, \ldots, u_{\dim \optsub' - \din})$. Then the lowest principal angles between $\overline{\optsub}_{\min}$ and $\optsub'$ are all equal to zero and the principal vectors in $\overline{\optsub}_{\min}$ are in $\vect(u_1, \ldots, u_{\dim \optsub' - \din})$. Consequently, the $\din$ largest principal angles are exactly the principal angles between $\optsub$ and $\optsub'$, which are also the principal angles between $\optsub$ and the $\din$-dimensional subspace that achieves the second minimum. Because the three distances can be expressed in term of principal angles, the equality holds.
\end{proof}

\subsection{Some results on the orthogonal group}

\begin{lemma}\label{lem:procruste}
	Let $d \in \N^*$ and let $(e_1, \ldots, e_d)$ be an orthonormal basis of $\R^d$. For $b \leq d$, let $(g_1, \ldots, g_b) \in \R^{d \times b}$ be a collection of orthonormal vectors in $\R^d$ such that 
	\[ \nor{e_i - g_i}{} \leq \eps, \quad \text{for all } i \in \eint{1, b}. \]
	Then we can complete this collection to obtain an orthonormal basis $(g_1, \ldots, g_d)$ of $\R^d$ satisfying
	\[ \nor{e_j - g_j}{} \leq 2\sqrt{b} \cdot \eps, \quad \text{for all } j \in \eint{1, d}.\]
\end{lemma}
\begin{proof}[Proof of Lemma \ref{lem:procruste}]
	We denote by $F$ the subspace $\vect(g_1, \ldots, g_b)$. Let us determine the distance between a vector $e_j$ and its orthogonal projection on $F^\perp$, for $j \in \eint{b+1, d}$. By Cauchy-Schwartz inequality, we have
	\[ \abs{\inner{e_j}{g_i}} \ \leq \ \nor{e_j}{} \nor{g_i - e_i} \ \leq \ \eps, \]
	for all $i \in \eint{1, b}$. Then
	\begin{equation}\label{eq:procruste}
		\nor{e_j - P_{F^\perp} (e_j)}{} \ = \ \nor{P_{F}(e_j)}{} = \pat{\sum_{i = 1}^b \inner{e_j}{g_i}^2 \nor{g_i}{}^2}^{1/2} \ \leq \ \sqrt{b}\cdot \eps.
	\end{equation}
	Thus the problem reduces to find a family of $d-b$ orthonormal vectors in $F^\perp$ with elements as close as possible to the vectors $P_{F^\perp}(e_j)$, for $j \in \eint{b+1, d}$. This is related to what is known as \textit{procruste} problem. 
	We denote by $A$ the matrix $A := \pat{P_{F^\perp}(e_{b+1}) | \cdots |  P_{F^\perp}(e_{d}) } \in \R^{d\times d-b}$ and we use Theorem 4.1 stated in \cite{Higham89MNPA}:

	\begin{theorem}[\cite{Higham89MNPA}]\label{thm:procruste}
		If $A$ admits a polar decomposition $A = UH$, and if $Q \in \R^{d \times d-b}$ has orthonormal columns, then
		\[  \vertiii{A-U}_2 \ \leq \ \vertiii{A - Q}_2. \]
	\end{theorem}
	Let us show that the columns of $U$ can be chosen in $F^\perp$. A singular value decomposition of $A$ can be written, $A = W D \transp{V}$, where $W$ has orthonormal columns, $V \in \O_{d-b}$, and $D \in \R^{d-b \times d-b}$ is diagonal. Therefore, $A = (W\transp{V}) VD\transp{V}$. Taking $U := W \transp{V} $ and $H := V D \transp{V}$, we have the polar decomposition $A = UH$ where $U$ has orthonormal columns. Because $\ima(A) = \vect \pat{P_{F^\perp}(e_j), j\in \eint{b+1, d}} \subset F^\perp$, it is possible to choose $W$ with columns in $F^\perp$, whence the desired result. 
	
	Now, taking $Q = \pat{e_{b+1} | \cdots | e_d}$, we have, for all unit vector $x \in \R^{d-b}$, 
	\[ P_{F^\perp}(Qx) \ = \ Ax. \]
	Moreover, using that $\abs{\inner{Qx}{g_i}} \leq \nor{Qx}{} \nor{g_i - e_i}{} \leq \eps$ for all $i \in \eint{1, b}$, we finally have
	\[ \nor{Qx - Ax}{}^2 \ = \ \nor{Qx - P_{F^\perp}(Qx)}{}^2 \ \leq \ b\eps^2, \]
	thus $\vertiii{A-Q} \leq \sqrt{b}\cdot \eps$.
	According to Theorem \ref{thm:procruste}, the last inequality is also true if we replace $Q$ by $U$. Because the columns $u_{b+1}, \ldots, u_d$ of $U$ are in $F^\perp$, the family $(g_1, \ldots, g_b, u_{b+1}, \ldots, u_d)$ is orthonormal and moreover satisfies \eqref{eq:procruste} by the triangle inequality.
\end{proof}

\bigskip

\begin{notation}
	Let $b, d \in \N^*$ with $b < d$ and let $\Bon{d}{b}$ be the set of all $b$-tuples of orthonormal vectors in $\R^d$.
\end{notation}

\begin{lemma}\label{lem:net}
	Let $b, d \in \N^*$ with $b \leq d$ and $0 < \eps \leq 1$. Then there exists a set $\G \subset \Bon{d}{b}$ such that for all $e \in \Bon{d}{b}$, there exists $g \in \G$ such that 
	\[ \max_{i \in \eint{1, b}} \nor{e_i - g_i}{2} \ \leq \ \eps \qquad \text{and} \qquad \abs{\G} \ \leq \ \pat{\frac{\pi d}{2}}^{b/2} \pat{\frac8\eps}^{b(d-1)}. \]
\end{lemma}
\begin{proof}[Proof of Lemma \ref{lem:net}]
	Let us construct $\G$. Let $\mathcal{T}$ be a set of balls in $\R^d$ with radius $\eps/2$ which cover $\mathbb{S}^{d-1}$ and such that $\abs{\mathcal{T}} = N(\mathbb{S}^{d-1}, \eps/2, \nor{\cdot}{2})$. We denote by $\overline{\mathcal{T}}^b$ the set of $b$-tuples of balls $(B_1, \ldots, B_b) \in \mathcal{T}^b$ such that $B_1 \times \cdots \times B_b$ contains at least one element of $\Bon{d}{b}$. Then, for each $e \in \Bon{d}{b}$, there exists $(B_1, \ldots, B_b) \in \overline{\mathcal{T}}^b$ such that $e \in B_1 \times \cdots \times B_b$.
	For each $B \in \overline{\mathcal{T}}^b$, choose one particular $b$-tuple $g \in \Bon{d}{b}$ such that $g \in B$ and let $\G$ be the set of these $b$-tuples when $B$ runs through $\overline{\mathcal{T}}^b$. It is clear that $\G$ satisfy the first condition of the lemma.
	Moreover, 
	\[ \abs{\G} \ = \ \bigabs{\overline{\mathcal{T}}^b} \ \leq \ \abs{\mathcal{T}^b} \ = \ N\pat{\eps/2, \mathbb{S}^{d-1}, \nor{\cdot}{2}}^b. \]
	Let us estimate the last quantity. We use the inequality 
	\[ N\pat{\eps, \mathbb{S}^{d-1}, \nor{\cdot}{2}} \ \leq \ D\pat{\eps, \mathbb{S}^{d-1}, \nor{\cdot}{2}}, \]
	where $D\pat{\eps, \mathbb{S}^{d-1}, \nor{\cdot}{2}}$ is the maximum number of disjoint balls with radius $\eps /2$ and with center in $\mathbb{S}^{d-1}$. Recall that 
	\[ \mathcal{A} \pat{\mathbb{S}^{d-1}} \ = \ \frac{2\pi^{d/2}}{\Gamma (d/2)} \qquad \text{and} \qquad \mathcal{V}\pat{B_{d-1}(\eps)} \ = \ \frac{\pi^{\frac{d-1}{2}} \eps^{d-1}}{\Gamma \pat{\frac{d+1}{2}}}.\]
	Consider the measure $\nu (\eps/2)$ of the hyperspherical cap defined by the intersection of $\mathbb{S}^{d-1}$ and a ball with center in $\mathbb{S}^{d-1}$ and with radius $\eps/2$.
	The colatitude angle of the cap is $\phi = 2 \arcsin (\eps/4)$ and, according to \cite{Li2011Hspcap}, 
	\[ \nu (\eps/2) \ = \ \frac{(d-1)\pi^{\frac{d-1}{2}}}{\Gamma \pat{\frac{d+1}{2}}} \int_0^\phi \sin^{d-2} (\theta) d\theta. \]
	Since $\phi \geq \eps/2$, 
	\[ \int _0^\phi \sin^{d-2} (\theta) d\theta \ \geq \ \int _0^\phi \pat{ \frac{\sin \phi}{\phi} \cdot \theta}^{d-2} d\theta  \ = \ \pat{\frac{\sin \phi}{\phi}}^{d-2} \frac{\phi^{d-1}}{d-1} \ \geq \ \pat{\frac{\sin \phi}{\phi}}^{d-2} \frac{1}{d-1} \pat{\frac{\eps}{2}}^{d-1} \]
	and, using the facts that $\eps \leq 1$, $\phi \leq \eps$, and $(\sin \phi) / \phi \geq 1/2$, we have

	\[ D\pat{\eps, \mathbb{S}^{d-1}, \nor{\cdot}{2}} \ \leq \ \frac{\mathcal{A} (\mathbb{S}^{d-1})}{\nu (\eps/2)} \ < \ \frac{\mathcal{A} (\mathbb{S}^{d-1})}{\mathcal{V} \pat{B_{d-1} (\eps/2)}\cdot \pat{\frac12}^{d-2}} \ = \ \sqrt{\pi} \cdot  \pat{\frac{4}{\eps}}^{d-1} \cdot \frac{\Gamma \pat{\frac{d+1}{2}}}{\Gamma (d/2)}.\]
	The ratio of two Gamma functions can be bounded as follows 
	\[ \sqrt{x + 1/4} \ < \ \frac{\Gamma (x+1)}{\Gamma (x + 1/2)} \ < \ \sqrt{x + 1/2}, \]
	for $x > -1/2$ (see \cite{Watson1959Gamma} and \cite{Luo&Qi2012Gamma}, Section 2.3).
	Choosing $x = (d-1)/2$, we obtain 
	\[ N\pat{\eps, \mathbb{S}^{d-1}, \nor{\cdot}{2}} \ < \ \sqrt{\frac{\pi d}{2}} \cdot \pat{\frac4\eps}^{d-1},\]
	hence the result.
\end{proof}

\bigskip

Let $b, d \in \N^*$ with $b < d$ and $q^* \in \O_d$. For $\eps > 0$, we define
\[ \Q_{q^*, \eps} \ := \ \{ q \in \O_d : \exists q' \in \O_d((q^*)^{-1}(E_{\b})), \ \vertiii{q^* q' -q} \leq \eps \}. \]
When establishing the prior mass condition, we need a lower bound on the probability $\Prob (\Theta \in \Q_{q^*, \eps})$ where $\Theta$ is the unit Haar measure on $\O_d$.
Since the Haar measure is translation invariant, it is sufficient to cover $\O_d$ with translations of $\Q_{q^*, \eps}$, that is, to cover $\O_d$ with sets $\overline{q} \Q_{q^*, \eps}$ where $\overline{q}$ belongs to some net $\mathcal{R} \subset \O_d$ and then remark that $\Prob (\Theta \in \Q_{q^*, \eps}) \geq 1/\abs{\mathcal{R}}$.

\begin{lemma}\label{lem:measure1}
	We have, 
	\[ \Prob (\Theta \in \Q_{q^*, \eps}) \ \geq \ \pat{\frac{2}{ \pi d}}^{\tfrac{b}{2}} \cdot \pat{\frac{\eps}{16\sqrt{bd}}}^{b (d-1)}.\]
\end{lemma}
\begin{proof}[Proof of Lemma \ref{lem:measure1}]
	Let $q'' \in \O_d$. The first step consists in constructing a net $\mathcal{R} \subset \O_d$ such that there exist $\overline{q} \in \mathcal{R}$ and $q \in \Q_{q^*, \eps}$ with $q'' = \overline{q} q$.
	Let $(u_1, \ldots, u_b, u_{b+1}, \ldots, u_d)$ be an orthonormal basis adapted to the direct sum $\R^d = (q^*)^{-1} (E_{\b}) \ \overset{\perp}{\bigoplus} \ (q^*)^{-1} (E_{1-\b})$.
	
	For all $b$-tuple of orthonormal vectors $g = (g_1, \ldots, g_b)$, we fix $r_g \in \O_d$ an isometry such that $r_g (q^* u_i) = g_i$ for all $i \in \eint{1, b}$.
	Moreover, we denote by $\G$ a set of $b$-tuples of orthonormal vectors in $\R^d$ such that, for all $b$-tuples $f = (f_1, \ldots, f_b)$ of orthonormal vectors, there exists $g \in \G$ satisfying
	\[ \sup _{i \in \eint{1, b}} \nor{g_i - f_i}{} \ \leq \ \frac{\eps}{2 \sqrt{b d}}. \]
	We claim that we can take $\mathcal{R} := \{ r_g : g \in \G\}$. Indeed, there exists $g \in \G$ such that
	\[ \sup _{i \in \eint{1, b}} \nor{g_i - q''(u_i)}{} \ \leq \ \frac{\eps}{2 \sqrt{b d}}. \]
	By Lemma \ref{lem:procruste}, we can extend $g$ in an orthonormal basis of $\R^d$ such that
	\begin{equation}\label{eq:basis1}
		\sup _{j \in \eint{1, d}} \nor{g_j - q''(u_j)}{} \ \leq \ \frac{\eps}{\sqrt{d}}.
	\end{equation}
	Then, writing $\overline{q} = r_g$ and taking $q$ such that $q(u_j) = r_g^{-1} (q'' u_j)$ for all $j \in \eint{1, d}$, we have $q'' = \overline{q}q$.
	Moreover, because $r_g^{-1} (g_j) \in E_{1-\b}$ and $(q^*)^{-1} r_g^{-1} (g_j) \in (q^*)^{-1} (E_{1-\b})$ for $j \in \eint{b+1, d}$, we can define $q'$ such that 
	\[ \begin{cases} 
		q'(u_i) = u_i, & \text{if } i \in \eint{1, b}, \\
		q'(u_j) = (q^*)^{-1} r_g^{-1} (g_j), & \text{if } j \in \eint{b +1, d}.
	\end{cases} \]
	Then, we have $q' \in \O_{d}((q^*)^{-1}(E_{\b}))$ and according to \eqref{eq:basis1}, 
	\[ \nor{q^* q'(u_i) - q(u_i)}{} \ = \ \nor{q^* (u_i) - r_g^{-1} (q'' u_i)}{} \ = \ \nor{r_g q^* (u_i) - q''(u_i)}{} \ \leq \ \frac{\eps}{\sqrt{d}}, \qquad \text{for } i \in \eint{1, b}, \]
	and, 
	\[ \nor{q^* q'(u_j) - q(u_j)}{} \ = \ \nor{r_g^{-1} (g_j) - q(u_j)}{} \ = \ \nor{g_j - r_g (q u_j)}{} \ \leq \ \frac{\eps}{\sqrt{d}}, \qquad \text{for } j \in \eint{b +1, d}. \]
	So $\vertiii{q^* q' - q} \leq \eps$ and the net $\mathcal{R} := \{ r_g : g \in \G \}$ is appropriate.
	Finally, by taking $\G$ as in Lemma \ref{lem:net}, we obtain
	\[ \abs{\mathcal{R}} \ \leq \ \pat{\frac{\pi d}{2}}^{\tfrac{b}{2}} \cdot \pat{\frac{16\sqrt{b d}}{\eps}}^{b (d-1)}, \]
	hence the result.
\end{proof}

\bigskip

\begin{lemma}\label{lem:netEntropy}
	Let $b, d \in \N^*$ be two integers such that $b < d$. For $\zeta > 0$, there exists a net $\mathcal{R}$ over $\O_d$ such that
	\[ \bigcup _{\overline{q} \in \mathcal{R}} \mathcal{A}_{\overline{q}} \ = \ \O_d, \]
	where 
	\[ \mathcal{A}_{\overline{q}} \ := \ \{ q \in \O_d \ | \ \exists q' \in \O_d (q^{-1} (E_{\b})), \ \vertiii{qq' - \overline{q}} \leq \zeta \}, \]
	and such that 
	\[ \abs{\mathcal{R}} \ \leq \ \pat{\frac{\pi\sqrt{bd}}{2}}^b \pat{\frac{16\sqrt{bd}}{\zeta}}^{b(d + b-2)}.\]
\end{lemma}
\begin{proof}
	Firstly, we remark that 
	\[ \mathcal{A}_{\overline{q}} \ = \ \bigacc{q \in \O_d \ | \ \exists q'' \in \O_d, \ q''_{|q^{-1}(E_{\b})} = q_{|q^{-1}(E_{\b})} \ \text{and} \ \vertiii{q'' - \overline{q}} \leq \zeta }. \]
	Thus, for $q \in \O_d$, we aim to construct $\overline{q}$ such that there exists $q'' \in \O_d$ satisfying $q''_{|q^{-1}(E_{\b})} = q_{|q^{-1}(E_{\b})}$ and $\vertiii{q'' - \overline{q}} \leq \zeta$.
	
	Let $(u_1, \ldots, u_b, u_{b+1}, \ldots, u_d)$ be an orthonormal basis adapted to the direct sum $\R^d =$ $ (q)^{-1} (E_{\b}) \ $ $\overset{\perp}{\bigoplus} \ (q)^{-1} (E_{1-\b})$.
	We introduce $\mathcal{F}$ a set of orthonormal basis of $E_{\b}$ such that, for all orthonormal basis $f'$ of $E_{\b}$, there exists $f \in \mathcal{F}$ such that 
	\[ \sup _{i \in \eint{1, b}} \nor{f_i - f'_i}{} \ \leq \ \frac{\zeta}{2\sqrt{bd}}, \]
	and we reuse the set $\G$ of Lemma \ref{lem:measure1}, replacing $\eps$ by $\zeta$.
	For all $g \in \G$ and $f \in \mathcal{F}$, we fix an isometry $r_{g, f} \in \O_d$ such that $r_{g, f}(g_i) = f_i$, for all $i \in \eint{1, b}$.
	
	By construction, there exist $f \in \mathcal{F}$ and $g \in \G$ such that 
	\[ \sup _{i \in \eint{1, b}} \nor{f_i - q(u_i)}{} \ \leq \ \frac{\zeta}{2\sqrt{bd}} \quad \text{and} \quad \sup _{i \in \eint{1, b}} \nor{g_i - u_i}{} \ \leq \ \frac{\zeta}{2\sqrt{bd}}. \]
	Then we choose $\overline{q} = r_{g, f}$. Using Lemma \ref{lem:procruste}, we extend $g$ to an orthonormal basis over $\R^d$ such that $\sup _{j \in \eint{1, d}} \nor{g_j - u_j}{} \ \leq \ \zeta/\sqrt{d}$ and we define $f_j := r_{g, f}(g_j) \in E_{\b}^\perp$, for $j \in \eint{b+1, d}$. Now we choose $q'' \in \O_d$ such that 
	\[ \begin{cases} 
		q''(u_i) = q(u_i) & \text{if } i \in \eint{1, b}, \\
		q''(u_j) = f_j & \text{if } j \in \eint{b +1, d}.
	\end{cases} \]
	This leads to $\nor{q''(u_j) - \overline{q}(u_j)}{} \leq \zeta/\sqrt{d}$, for all $j \in \eint{1, d}$, hence $\vertiii{q'' - \overline{q}} \leq \zeta$.
	We can thus define the net $\mathcal{R}$ as the set of all isometries $r_{g, f}$ for $g \in \G$ and $f \in \mathcal{F}$. According to Lemma \ref{lem:net}, this yields the upper bound
	\begin{align*}
		\abs{\mathcal{R}} \ & = \ \abs{\G} \cdot \abs{\mathcal{F}} \\[0.2cm]
		& \leq \ \pat{\frac{\pi d}{2}}^{\tfrac{b}{2}} \pat{\frac{16\sqrt{bd}}{\zeta}}^{b(d-1)} \pat{\frac{\pi b}{2}}^{\tfrac{b}{2}} \pat{\frac{16\sqrt{b d}}{\zeta}}^{b(b-1)}. \qedhere
	\end{align*}
\end{proof}

\smallskip



\renewcommand{\mkbibnamefamily}[1]{\textsc{#1}} 

\renewbibmacro*{volume+number+eid}{%
	\printfield{volume}%
	\printfield{number}%
	\setunit{\addcomma\space}%
	\printfield{eid}}
\DeclareFieldFormat[article]{number}{\textbf{\mkbibparens{#1}}}
\DeclareFieldFormat[article]{volume}{\textbf{#1}}

\printbibliography

\end{document}